\documentclass [reqno] {amsart}

\copyrightinfo{2010}{American Mathematical Society}

\usepackage{amsmath}
\usepackage{color}
\usepackage{hyperref}

\newtheorem{theorem}{Theorem}[section]
\newtheorem{proposition}[theorem]{Proposition}
\newtheorem{lemma}[theorem]{Lemma}
\newtheorem{corollary}[theorem]{Corollary}

\theoremstyle{definition}

\theoremstyle{remark}
\newtheorem{remark}[theorem]{Remark}

\numberwithin{equation}{section}

\begin{document}

\title[Inverse image of precompact sets]
      {Inverse image of precompact sets and existence theorems for    
			the Navier-Stokes equations in spatially periodic setting
			%An open mapping theorem for
      % regular spatially periodic solutions to
      % the Navier-Stokes equations %over Bochner-Sobolev spaces
			}

\author{A. Shlapunov}

\address[Alexander Shlapunov]
        {Siberian Federal University
\\
         Institute of Mathematics and Computer Science
\\
         pr. Svobodnyi 79
\\
         660041 Krasnoyarsk
\\
         Russia}

\email{ashlapunov@sfu-kras.ru}

\author{\fbox{N. Tarkhanov}}

\address[\fbox{Nikolai Tarkhanov}]
        {Universit\"{a}t Potsdam
\\
         Institut f\"{u}r Mathematik
\\
         Karl-Liebknecht-Str. 24/25
\\
         14476 Potsdam (Golm)
\\
         Germany}

\email{tarkhanov@math.uni-potsdam.de}

\subjclass [2010] {Primary 76N10; Secondary 35Q30, 76D05}

\keywords{Navier-Stokes equations,
          smooth solutions,
          existence theorem}

\begin{abstract}
We consider the initial problem for the Navier-Stokes equations over
${\mathbb R}^3 \times [0,T]$ with a positive  time $T$ in the spatially periodic setting.
Identifying periodic vector-valued functions on  ${\mathbb R}^3$ with functions on the
$3\,$-dimensional torus ${\mathbb T}^3$, we prove that the problem 
induces an open injective mapping ${\mathcal A} _s:  B^{s}_1 \to B^{s-1}_2$ where 
$B^{s}_1$, $B^{s-1}_2$ are elements from scales of specially constructed function spaces of 
Bochner-Sobolev type parametrized with the smoothness index $s \in \mathbb N$. Finally, we 
prove rather expectable statement that a map ${\mathcal A} _s$ 
is surjective if and only if the inverse image ${\mathcal A} _s ^{-1}(K)$ 
of any precompact set $K$ from the range of the map ${\mathcal A} _s $ is bounded 
in the Bochner space $L^{\mathfrak s} ([0,T], L ^{\mathfrak s} ({\mathbb T}^3))$ 
with the Ladyzhenskaya-Prodi-Serrin numbers ${\mathfrak s}$, ${\mathfrak r}$. 
\end{abstract}

\maketitle

\section*{Introduction}
\label{s.0}

The problem of describing the dynamics of incompressible viscous fluid is of great importance 
in applications.
The dynamics is described by the Navier-Stokes equations and the problem consists in finding a
sufficiently regular solution to the equations for which a uniqueness theorem is available,
   cf. \cite{Lady03}.
Essential contributions has been published in the research articles
   \cite{Lera34a,Lera34b},
   \cite{Kolm42},
   \cite{Hopf51},
as well as surveys and books
   \cite{Lady70}),
   \cite{Lion61,Lion69},
   \cite{Tema79},
   \cite{FursVish80},
etc.

More precisely, let
   $\varDelta = \partial^2_{x_1} + \partial^2_{x_2} + \partial^2_{x_3}$ be the Laplace 
operator,
   $\nabla$ and $\mathrm{div}$ be the gradient operator and the divergence operator, 
respectively,
in the Eucledean space ${\mathbb R}^3$.
In the sequel we consider the following initial problem.
Given any sufficiently regular vector-valued functions
   $f = (f^1, f^2, f^3)$ and
   $u_0 = (u_{0}^1, u_{0}^2, u_{0}^3)$
on
   ${\mathbb R}^3 \times [0,T]$ and
   ${\mathbb R}^3$,
respectively, find a pair $(u,p)$ of sufficiently regular functions
   $u = (u^1, u^2, u^3)$ and
   $p$
on ${\mathbb R}^3 \times [0,T]$ satisfying
\begin{equation}
\label{eq.NS}
\left\{
\begin{array}{rcll}
   \partial _t u   -  \mu \varDelta u + (u \cdot \nabla) u  + \nabla p
 & =
 & f,
 & (x,t) \in {\mathbb R}^3 \times (0,T),
\\
   \mbox{div}\, u
 & =
 & 0,
 & (x,t) \in {\mathbb R}^3 \times (0,T),
\\[.05cm]
   u
& =
& u_0,
& (x,t) \in \mathbb{R}^3 \times \{ 0 \}
\end{array}
\right.
\end{equation}
with positive fixed numbers $T$ and $\mu$.
We additionally assume that the data $f$ and $u_0$ are spatially periodic with a period $\ell > 0$,
i.e.,  for any $1 \leq j \leq 3$ we have
\begin{equation*}
\begin{array}{rcl}
   f (x+\ell e_j,t)
 & =
 & f (x,t),
\\
   u_0 (x+\ell e_j)
 & =
 & u_0 (x)
\end{array}
\end{equation*}
whenever $x \in {\mathbb R}^3$ and $t \in [0,T]$, where $e_j$ is as usual the $j\,$-th 
unit basis vector in ${\mathbb R}^3$. Then, the solution $(u,p)$ is also looked for in the 
space of spatially periodic functions with period $\ell$ on ${\mathbb R}^3 \times [0,T]$.
Relations \eqref{eq.NS} are usually referred to as but 
the Navier-Stokes equations for incompressible fluid  with given
   dynamical viscosity 
	$\mu$ of the fluid under the consideration,
   density vector of outer forces $f$, 
   the initial velocity $u_0$
and the search-for velocity vector field $u$ and the pressure $p$ of the flow, see for instance
   \cite{LaLi},
   \cite{Tema79}
for the classical setting or
   \cite{Serr59b},
   \cite{Tema95}
for the periodic setting.

In the present paper we use the method of energy type  estimates  to obtain 
an open mapping theorem and a criterion of the surjectivity for the mapping induced by 
\eqref{eq.NS} over scales of specially constructed function spaces of 
Bochner-Sobolev type parametrized with the smoothness index $s \in \mathbb N$. 

After Leray \cite{Lera34a,Lera34b}, a great attention of researchers was paid to weak 
solutions of 
\eqref{eq.NS} in cylindrical domains in ${\mathbb R}^3 \times [0,+\infty)$.
E. Hopf \cite{Hopf51} proved the existence of weak solutions to \eqref{eq.NS} satisfying 
reasonable estimates. However, in this full generality no uniqueness theorem for a weak 
solution has been known. On the other hand, under stronger conditions on the solution, it is 
unique, cf. results \cite{Lady70}, \cite{Lady03} by O.A. Ladyzhenskaya who proved the 
existence of a smooth solution for the 
two-dimensional version of problem \eqref{eq.NS}. Some authors (see, e.g., Leray 
\cite{Lera34a,Lera34b} and also a recent paper \cite{Tao16} by T. Tao) 
expressed a lot of skepticism on the existence of regular solutions for all regular data in 
${\mathbb R}^3$. Thus, let us explain the place of our investigation in a great number of 
works on the subject.

Traditionally, the approach based on a priori estimates is used for the proof of existence 
theorems for boundary value problems in  the theory of parabolic and elliptic equations.
A solution is often given as a formal series of Faedo-Galerkin type.
Then a proper a priori estimate (for the Navier-Stokes equations this is the so-called energy 
estimate) provides the convergence of the series to a generalised (weak) solution.
On the next step, modified a priori estimates usually lead to a uniqueness theorem and to an 
improvement of the regularity of the weak solution up to an acceptable level depending on the 
regularity of the data (see for instance the books \cite{Lady70}, \cite{Tema79} and the 
references given there). For the Navier-Stokes equations this method was only partially 
efficient: the existence of a weak solutions was proved but no uniqueness theorem or modified 
a priori estimate improving the regularity of weak solutions
have been found up to now.

The main technical difficulty appears at the point where the energy type (in)equ\-alities are 
used for finite indices of smoothness and integrability: a nonnegative homogeneous functional 
(a norm or a quasinorm) stands on the left hand side while on the
right hand side the nonlinearity and the summands depending on the data are placed.
In the three-dimensional case the nonlinearity apparently 
can not be estimated via the left hand hand side with respect to the smoothness or 
integrability with a universal constant for all solutions and data. As is known, the corresponding 
excess, expressed additively in terms of the 
order of generalised partial derivatives with respect to space variables from the Lebesque 
space $L^2 ({\mathbb T}^3)$, equals $1/4$. This fact hints that one should 
involve infinitely differentiable data and all their derivatives 
in order to compensate the excess. 
 In this case it is not so easy to use the advantage
 of  high regularity of the data $(f,u_0)$ because the series
involving their partial derivatives are not always converge
   (cf. \cite{FoTe89} where Gevrey quasianalytic classes were employed).
	
But the integrability is a kind of regularity, too.
Due to
   \cite{Pro59},
   \cite{Serr62},
   \cite{Lady70} and
   \cite{Lion61,Lion69},
it is known that the uniqueness theorem and improvement of regularity actually follow from 
the existence of a weak solution in the Bochner class 
$L^\mathfrak{s} ([0,T], L^\mathfrak{r} ({\mathbb R^3}))$  with the Ladyzhenskaya-Prodi-Serrin 
numbers $\mathfrak{r} $, $\mathfrak{s}$ satisfying  
   $2/\mathfrak{s} + 3/\mathfrak{r} = 1$ and
   $\mathfrak{r} > 3$
(the limit case $r = 3$ was added to the list in \cite{ESS03}).
On the other hand, the standard energy estimate provides the existence in 
$L^{s} ([0,T], L^{r} ({\mathbb R^3}))$ with $2/s + 3/r = 3/2$, only.

We note that the existence of regular solutions  to the Navier-Stokes 
equations for sufficiently small data in different 
spaces  is known since J.Leray. In addition 
to these results
O.~A. Ladyzhenskaya discovered  the so-called stability property for the Navier-Stokes 
equations in some Bochner type spaces (see \cite[Ch.~4, \S~4, Theorems 10 and 11]{Lady70}).
Namely, if for sufficiently regular data $(f,u_0)$  there is a sufficiently regular solution 
$(u,p)$ to the Navier-Stokes equations, then there is a neighbourhood of the data in which 
all elements admit solutions with the same regularity.

Let us indicate the most important points of our approach. 

1) We treat the Navier-Stokes equations 
within the framework of theory of operator equations associating them with  
a continuous nonlinear mapping  between Banach spaces. 

2) We avoid weak solutions to the Navier-Stokes equations sharing the opinion that there might 
be nonunique weak solutions with singularities for some insufficiently regular data. 
More precisely, we define two scales $\{B^{s}_1\}_{s \in {\mathbb Z}_+}$,   
$\{B^{s}_2\}_{s \in {\mathbb Z}_+}$  of separable Banach spaces such that 
\begin{itemize}
\item
each space of the scale $\{B^{s}_1\}_{s \in {\mathbb Z}_+}$ is continuously 
embedded into the spaces $L^{s} ([0,T], L^{r} ({\mathbb R^3}))$ with $2/s + 3/r = 3/2$,
\item
 the Navier-Stokes equations induce non-linear continuous mappings 
${\mathcal A}_s:  B^{s}_1 \to B^{s-1}_2$ for all $s \in {\mathbb N}$, 
\item
the components of vector fields belonging to the intersections 
$\cap_{s=1}^\infty B^{s}_1$, $\cap_{s=0}^\infty B^{s}_2$ are 
infinitely differentiable functions on the torus ${\mathbb T}^3$. 
\end{itemize}
Actually, the elements of the chosen spaces belong to 
the Bochner-Sobolev spaces over ${\mathbb T}^3 \times [0,T]$ 
of rather high Sobolev smoothness. 
In particular, this means that the solutions to the corresponding 
operator equations are at least the so-called {\it strong} solutions 
to the standard weak setting of the Navier-Stokes equations, 
see \cite[Ch. 3, \S 3.6]{Tema79}; of course, 
for sufficiently large $s$ their components  are differentiable functions with 
the smoothness depending upon the Sobolev Embedding Theorems. 

3) We use in full the mentioned above stability propertyfor the Navier-Stokes 
equations discovered by O.~A. Ladyzhenskaya.  We extend this property 
to the mappings ${\mathcal A}_s:  B^{s}_1 \to B^{s-1}_2$ with arbitrary $s \in {\mathbb N}$, 
expressing it as open mapping theorem for \eqref{eq.NS}. The statement is also 
adapted for the Fr\'chet spaces of smooth vector fields over the torus. 

4) At the next step we use the standard topological arguments immediately implying that a 
nonempty open connected set in a topological vector space coincides 
with the space itself if and only if the set is closed. In order to prove that the set of 
data admitting regular solutions to the Navier-Stokes Equations is closed, we do  not use the 
Faedo-Galerkin formal series replacing them by real approximate solutions to the Navier-
Stokes equations. More precisely, it appears that in the chosen function spaces the 
closedness of the image is equivalent to the boundedness 
of the sequences in the preimage corresponding to sequences 
converging to an element of the image's closure.

5) Finally, we prove that a map ${\mathcal A} _s$ 
is surjective if and only if the inverse image ${\mathcal A} _s ^{-1}(K)$ 
of any precompact set $K$ from the range of the map ${\mathcal A} _s $ is bounded 
in the Bochner space $L^{\mathfrak s} ([0,T], L ^{\mathfrak s} ({\mathbb T}^3))$ 
with the Ladyzhenskaya-Prodi-Serrin numbers ${\mathfrak s}$, ${\mathfrak r}$.
This echoes the idea of using the properness property
to study nonlinear operator equations, see for instance \cite{Sm65}. 

Now let us comment the contents of the paper.

Section \ref{s.prelim} contains the preliminary matters such as notation, properties of the 
involved function spaces and the classical results concerning weak solutions to the Navier-
Stokes equations.

Section \ref{s.OM} is devoted to an open mapping theorem for the 
Navier-Stokes equations treated as a continuous mapping 
on the scales of separable  Banach (Bochner-Sobolev type) spaces. 
In order to achieve it we consider the linearised problem related to the Fr\'echet derivative 
of the the Navier-Stokes mappings, cf. similar linear problems in
   \cite{Lady67} or
   \cite[Ch.~3, \S~1--\S~4]{LadSoUr67}.
	
Finally, in Section \ref{s.surjective}, using typical estimates for sufficiently regular 
solutions to the Navier-Stokes equations in Bochner-Sobolev spaces in the 
spatially periodic case
   (cf. \cite[Pt.~1, \S 3, \S~4]{Tema95}),
we prove the criterion for the existence theorems related to the Navier-Stokes 
equations acting on the introduced scales  of Banach spaces over  
${\mathbb T}^3 \times [0,T]$ to be true. Corollaries concerning $C^\infty$ smooth 
solutions are also discussed in this section.

\section{Preliminaries}
\label{s.prelim}

As usual, we denote by $\mathbb{Z}_{+}$ the set of all nonnegative integers including zero, 
and by $\mathbb{R}^n$ the Euclidean space of dimension $n \geq 2$ with coordinates 
$x = (x^1, \ldots, x^n)$.

In the sequel we use systematically the Gronwall lemma 
in the integral form for continuous functions and its generalizations.

\begin{lemma}
\label{l.Perov}
Let
   $0 < \gamma < 1$ and
   $A \geq 0$
be constants and let
   $B$,
   $C$ and
   $Y$
be nonnegative continuous functions defined on a segment $[a,b]$.
If moreover $Y$ satisfies the integral inequality
\begin{equation*}
   Y (t) \leq A  + \int_a^t (B (s) Y (s) + C (s) (Y (s))^{1-\gamma}) ds
\end{equation*}
for all $t \in [a,b]$, then
\begin{equation*}
   Y (t)
 \leq
   \left( A^{\gamma} \exp \Big( \gamma \int_a^t  B (s) ds \Big)
        + \gamma \int_a^t C (s) \exp \Big( \gamma \int_s^t B (t') d t' \Big) ds
   \right)^{1/\gamma}
\end{equation*}
for all $t \in [a,b]$.
\end{lemma}

\begin{proof} 
For $\gamma=1$ this gives 
Gronwall lemma,  see, for instance, \cite{Groe19} or \cite[p.~353]{MPF91}.
For $0 < \gamma < 1$ see \cite{Per59} or \cite[p.~360]{MPF91}.
\end{proof}

Also the  (discrete) Young inequality will be of frequent use in this paper.
To wit, given any $N = 1, 2, \ldots$, it follows that
\begin{equation}
\label{eq.Young}
   \prod_{j=1}^N a_j \leq \sum_{j=1}^N \frac{a_j^{p_j}}{p_j}
\end{equation}
for all positive numbers $a_j$ and all numbers $p_j \geq 1$ satisfying
   $\displaystyle \sum_{j=1}^N 1/p_j = 1$.

We continue with introducing proper function spaces.
For a measurable set $\sigma$ in ${\mathbb R}^n$ and $p\in [1, +\infty)$, we denote by 
$L^p (\sigma)$ the usual Lebesgue space of functions on $\sigma$.
When topologised under the standard norm $\| \cdot \|_{L^p (\sigma)} $ 
it is complete, i.e., a Banach space.
Of course, for $p = 2$ the norm is generated by the standard inner product
$(\cdot,\cdot)_{L^2 (\sigma)}$ and so $L^2 (\sigma)$ is a Hilbert space.
As usual, the scale $L^p (\sigma)$ continues to include the case $p = \infty$, too.
The integral H\"older inequality is one of the frequently used tools for us, to wit,
\begin{equation}
\label{eq.Hoelder}
   \| \prod_{j=1}^N a_j \|_{L^q (\sigma)}
 \leq
   \prod_{j=1}^N \| a_j \|_{L^{q_j} (\sigma)}
\end{equation}
for all $a_j \in L^{q_j} (\sigma)$, provided that
   $q \geq 1$,
   $q_j \geq 1$
and
   $\displaystyle \sum_{j=1}^N 1/q_j = 1/q$,
see for instance \cite[Corollary 2.6]{Ad03}.

For a  domain $\mathcal X$ in ${\mathbb R}^n$, we denote by
   $C^\infty_{\mathrm{comp}} ({\mathcal X})$
the set of all $C^\infty$ functions with compact support in $\mathcal X$.
If $s = 1, 2, \ldots$, we write $W^{s,p} ({\mathcal X})$ for the Sobolev space of all functions
$u \in L^p ({\mathcal X})$ whose generalised partial derivatives up to order $s$ belong to
$L^p ({\mathcal X})$. This is a Banach space with the standard norm 
$\| \cdot \|_{W^{s,p} ({\mathcal X})}$. 
Then
   $\mathring{W}^{s,p} ({\mathcal X})$
denotes the closure of the subspace $C^\infty_{\mathrm{comp}} ({\mathcal X})$ in $W^{s,p} ({\mathcal X})$.
The space $W^{s,p}_{\mathrm{loc}} ({\mathcal X})$ consists of functions belonging to $W^{s,p} (U)$ for
each relatively compact domain $U \subset {\mathcal X}$. 
As usual, in the case $p=2$ we simply write $H^{s} ({\mathcal X})$ instead of 
 $W^{s,2} ({\mathcal X})$.
This is a Hilbert space with the standard inner product 
$ (\cdot ,\cdot)_{H^{s} ({\mathcal X})}$. 
The scale of Sobolev spaces continues to include the case of negative $s$, too.
We will use only the space $H^{-s} ({\mathcal X})$ defined as the completion of
   $C^\infty_{\mathrm{comp}} ({\mathcal X})$
with respect to the norm
$$
   \| u \|_{H^{-s} ({\mathcal X})}
 = \sup_{{v \in C^\infty_{\mathrm{comp}} ({\mathcal X}) \atop  v \ne 0}}
   \frac{|(u,v)_{L^2 ({\mathcal X})}|}{\| v \|_{H^s ({\mathcal X})}}.
$$
It may be easily identified with the dual of $\mathring{H}^{s} ({\mathcal X})$, see for 
instance \cite[Theorem 3.12]{Ad03}.

Next, for    $s = 0, 1, \ldots$ and $0 \leq \lambda < 1$,
we denote by $C^{s,\lambda} (\overline{\mathcal{X}})$ 
the so-called H\"{o}lder spaces, see for instance
   \cite[Ch.~1, \S~1]{LadSoUr67},
   \cite[Ch.~1, \S~1]{Lady70}.
The normed spaces $C^{s,\lambda} (\overline{\mathcal{X}})$ with
   $s \in \mathbb{Z}_{+}$ and
   $\lambda \in [0,1)$
are known to be Banach spaces which admit the standard embedding theorems.

We are now ready to define proper spaces of periodic functions on ${\mathbb R}^3$.
For this purpose, fix any $\ell > 0$ and denote by ${\mathcal Q}$ be the cube $(0,\ell)^3$
of side length $\ell$.
Suppose $s \in \mathbb Z_{+}$.
We denote by $W^{s,p}$ the space of all functions
   $u \in W^{s,p}_\mathrm{loc} ({\mathbb R}^3)$
which satisfy the periodicity condition
\begin{equation}
\label{eq.period}
   u (x + \ell e_j) = u (x)
\end{equation}
for all
   $x \in {\mathbb R}^3$ and
   $1 \leq j \leq 3$,
where $e_j$ is the $j\,$-th unit basis vector in ${\mathbb R}^3$.
This is a Banach space with the norm
   $\| u \|_{W^{s,p}} := \| u \|_{W^{s,p} ({\mathcal Q})}$.
The space $H^{s}$ is obviously a Hilbert space endowed with the inner product
$$
   (u,v)_{H^{s}} = (u,v)_{H^{s} (\mathcal Q)}.
$$
The functions from $H^{s} $ can be easily characterised by their Fourier series
expansions with respect to the orthogonal system
   $\{ e^{\sqrt{-1} (k,z) (2 \pi / \ell)} \}_{k \in {\mathbb Z}^3}$
in $L^2 ({\mathcal Q})$.
Indeed, as the system consists of eigenfunctions of the Laplace operator $\varDelta$ 
corresponding to eigenvalues
   $\{ - (k,k) (2 \pi / \ell)^2 \}_{k \in {\mathbb Z}^3}$,
we see that the above scale of Sobolev spaces 
may be defined for all $s \in {\mathbb R}$ by
\begin{equation}
\label{eq.Hs.per}
   H^{s} 
 = \{  u = \sum_{k \in {\mathbb Z}^3} c_k (u) e^{\sqrt{-1} (k,z) (2 \pi / \ell)} :\,
      |c_0 (u)|^2 + \sum_{k \in {\mathbb Z}^3 \atop k \ne 0} (k,k)^{s} |c_k (u)|^2 < \infty
   \}.
\end{equation}
Clearly, the norm $\| \cdot \|_{H^{s} }$ is equivalent to the norm given by
$$
   \| u \|_{s}
 = \Big( \sum_{k \in {\mathbb Z}^3} (1 + (k,k))^s |c_k (u)|^2 \Big)^{1/2}.
$$
Traditionally, $\dot{H}^{s} $ stands for the subspace of $H^{s} $
consisting of the elements $u$ with $c_0 (u) = 0$ in \eqref{eq.Hs.per}.

Actually, this discussion leads us to the identification of the space $H^{s} $
with Sobolev functions on the torus ${\mathbb T}^3$, to wit,
   $H^{s}  \cong H^{s} ({\mathbb T}^3)$,
see \cite[\S~2.4]{Agra90} and elsewhere.

We also need an efficient tool for obtaining a priori estimates.
Namely, it is the Gagliardo-Nirenberg inequality, see \cite{Nir59} for functions on ${\mathbb R}^n$.
Its analogue for the torus reads for periodic functions as follows
   (see for instance \cite[\S~2.3]{Tema95}).
For $1 \leq  p \leq \infty$, set
\begin{equation}
\label{eq.Dj}
   \| \nabla^j u \|_{L^p ({\mathcal Q})}
 := \max_{|\alpha| = j} \| \partial^\alpha u \|_{L^p ({\mathcal Q})}.
\end{equation}
Then for any function $u \in L^{q_0}  \cap L^{s_0} $ satisfying
   $\nabla^{j_0} u \in L^{p_0} $  and
   $\nabla^{k_0} u \in L^{r_0} $
it follows that
\begin{equation}
\label{eq.L-G-N}
   \| \nabla^{j_0} u \|_{L^{p_0} (\mathcal{Q})}
 \leq
   c_1\, \| \nabla^{k_0} u \|^a_{L^{r_0} (\mathcal{Q})} \| u \|^{1-a}_{L^{q_0} (\mathcal{Q})}
 + c_2\, \| u \|_{L^{s_0} (\mathcal{Q})}
\end{equation}
whenever
   $s_0 \geq 1$ and
   $0 \leq a \leq 1$,
where
\begin{equation}
\label{eq.L-G-N.cond}
\begin{array}{rcl}
   \displaystyle
   \frac{1}{p_0}
 & = &
   \displaystyle
   \frac{j_0}{3} + a \left( \frac{1}{r_0} - \frac{k_0}{3} \right) + (1-a)\, \frac{1}{q_0},
\\
   \displaystyle
   \frac{j_0}{k_0}
 & \leq &
   a,
\end{array}
\end{equation}
the constants $c_1$ and $c_2$ depend on $j_0$, $k_0$, $s_0$, $p_0$, $q_0$ and $r_0$ but not
on $u$.

Next, for $s \in {\mathbb Z}_+$ and $\lambda \in [0,1)$, denote by $C^{s,\lambda} $
the space of all functions on ${\mathbb R}^3$ which belong to $C^{s,\lambda} (\overline{\mathcal X})$
for any bounded domain ${\mathcal X} \subset {\mathbb R}^3$ and satisfy \eqref{eq.period}.
The space $C^\infty $ of spatially periodic $C^\infty$-functions reduces to the intersection
of the spaces $C^{s,0} $ over $s \in \mathbb{Z}_0$.
It is endowed with the Fr\'echet topology given by the family of norms
   $\{ \| u \|_{C^{s,0} } \}_{s \in {\mathbb Z}_+}$.
Let $D'$ stand for the space of distributions on ${\mathbb T}^3$, i.e.,
the space of continuous linear functionals on the Fr\'echet space $C^\infty $ endowed
with the weak topology.

We will use the  symbol $\mathbf{L}^p$ for the space of periodic vector fields
$u = (u^1,u^2,u^3)$ on $\mathbb{R}^3$ with components $u_i$ in $L^p$.
The space is endowed with the norm
$$
   \| u \|_{\mathbf{L}^p}
 = \Big( \sum_{j=1}^3 \int_{{\mathcal Q}} |u^j (x)|^p dx \Big)^{1/p}.
$$
In a similar way we designate the spaces of periodic vector fields on $\mathbb{R}^3$ whose components
are of Sobolev or H\"{o}lder class.
We thus get
   ${\mathbf W}^{s,p} $,
   ${\mathbf H}^{s}$ and
   ${\mathbf C}^{s,\lambda}$,
respectively.
By
   ${\mathbf C}^\infty$ and
   ${\mathbf D}'$
are meant the spaces of infinitely smooth periodic vector fields or distribution vector 
fields on $\mathbb{T}^3$.

To continue, we recall basic formulas of vector analysis saying that
\begin{equation}
\label{eq.deRham}
\begin{array}{rclcrcl}
   \mathrm{rot}\, \nabla
 & =
 & 0,
 &
 & \mathrm{div}\, \nabla
 & =
 & \varDelta,
\\
   \mathrm{div}\, \mathrm{rot}
 & =
 & 0,
 &
 & - \mathrm{rot}\, \mathrm{rot} + \nabla\, \mathrm{div}
 & =
 & E_3 \varDelta
\end{array}
\end{equation}
where $E_3$ is the unit matrix of type $(3 \times 3)$.

Given any differential operator $A$ with $C^\infty$ coefficients on the space of vector
fields, we denote by $\ker (A)$ the subspace of $\mathbf{D}'$ consisting of all
vector fields satisfying $Au = 0$ in the sense of distributions in ${\mathbb R}^3$.
Furthermore, for an integer $s$, we write $V_s$ for the space
    ${\mathbf H}^{s} \cap \ker (\mathrm{div})$.
The designations $H$ and $V$ are usually used for $V_0$ and $V_1$, respectively,
   see \cite[\S~2.1]{Tema95}.
If $s = 1, 2, \ldots$, the dual $V'_s$ of $V_s$ can be identified with the completion of
${\mathbf C}^\infty  $ with respect to the `negative' norm
\begin{equation}
\label{eq.v.dual}
   \| u \|_{V'_s}
 = \sup_{{v \in V_s \atop  v \ne 0}}
   \frac{|(u,v)_{\mathbf{L}^2 }|}{\| v \|_{\mathbf{H}^{s}  }}.
\end{equation}

In order to characterize the space $V_s$ we denote by $\mathbb{N}_{2,3}$ the set of all natural
numbers that can represented as $(k,k) = k_1^2 + k_2^2 + k_3^2$, where
   $k = (k_1,k_2,k_3)$
is a triple of natural numbers.
For $m \in \mathbb{N}_{2,3}$, let $S_m$ be the finite-dimensional linear span of the system
   $\{ e^{\sqrt{-1} (k,z) (2 \pi / \ell)} \}_{(k,k) = m}$
of eigenfunctions of the Laplace operator, and let $\mathbf{S}_m$ be the corresponding space
of vector fields on $\mathbb{R}^3$. 

Thus, we arrive at the following useful statements, cf. also \cite{Saks04}.
\begin{lemma}
\label{l.basis.V}
Let $m \in {\mathbb N}_{2,3}$.
There are orthonormal bases
   $\{ v_{m,j} \}_{j=1}^{J_m}$ and
   $\{ w_{m,k} \}_{k=1}^{K_m}$
in the spaces
   $\mathbf{S}_m \cap \ker (\mathrm{div})$ and
   $\mathbf{S}_m \cap \ker (\mathrm{rot})$,
respectively, with respect to the unitary structure of $\mathbf{L}^2   $, such that
$$
\begin{array}{rcccl}
   \mathrm{rot} \circ \mathrm{rot}\, v_{m,j}
 & =
 & -\, \varDelta v_{m,j}
 & =
 & m (2 \pi / \ell)^2 v_{m,j},
\\
   \nabla \circ \mathrm{div}\, w_{m,k}
 & =
 & \varDelta w_{m,k}
 & =
 & -\, m (2 \pi / \ell)^2 w_{m,k}
\end{array}
$$
for all
   $j = 1, \ldots, J_m$ and
   $k = 1, \ldots, K_m$.
\end{lemma}
We introduce the Fourier coefficients of a field $u \in {\mathbf D}'   $ with
respect to the system $\{ e_1, e_2, e_3, v_{m,j} \}$:
$$
\begin{array}{rcl}
   c^{0,j} (u)
 & =
 & \langle u,e_j \rangle
/ (e_j,e_j)_{\mathbf{L}^2},
\\
   c^{m,j} (u)
 & =
 & \langle u,v_{m,j} \rangle .
\end{array}
$$

\begin{proposition}
\label{p.Vs}
Let $s \in \mathbb{Z}$.
The system
$
   \{ e_1, e_2, e_3, v_{m,j}, w_{m,k} \}
$
is an orthogonal basis in $\mathbf{H}^s   $.
In particular, the space $V_{s}$ consists of all vector fields $u \in \mathbf{H}^s   $
which satisfy
$$
   u
 = \sum_{j=1}^3 c_{0,j} (u) e_i
 + \sum_{m \in {\mathbb N}_{2,3}} \sum_{j=1}^{J_m} c^{m,j} (u) v_{m,j},
$$
where
$
   \displaystyle
   \sum_{m \in {\mathbb N}_{2,3}} \sum_{i=1}^{J_m}  m^{s}\, |c_{m,i} (u)|^2 < \infty.
$
\end{proposition}

It is worth pointing out that the orthogonality in the space $\mathbf{H}^s   $ refers to the
inner product
$$
   (u,v)_s
 = c_0 (u) \overline{c_0 (v)}
 + \sum_{\lambda_j \neq 0} \lambda_j^s\, c_j (u) \overline{c_j (v)},
$$
where $c_j (u)$ are the Fourier coefficients of $u$ with respect to an orthonormal system of eigenfunctions
of the Laplace operator in $\mathbf{L}^2   $ corresponding to the eigenvalues $\lambda_j$.

\begin{proof}
The arguments are straightforward.
Clearly, the series for $u$ is tacitly assumed to converge in the norm of $\mathbf{H}^s   $.
\end{proof}

Denote by $\mathbf{P}$ the orthogonal projection of $\mathbf{L}^2 $ onto $V_0$
which is usually referred to as the Helmholtz projection.
By Proposition \ref{p.Vs}, we 
get 
\begin{equation*}
   \mathbf{P} = \varPi + \mathrm{rot}^\ast \mathrm{rot}\, \varphi
\end{equation*}
where the operators
   $\varPi$ and
   $\varphi$
are given by
\begin{equation*}
\begin{array}{rcl}
   \varPi u
 & =
 & c_0 (u),
\\
   \varphi u
 & =
 & \displaystyle
   -\, \sum_{k \in \mathbb{Z}^3 \atop k \neq 0}
   \frac{c_k (u)}{(k,k) (2 \pi / \ell)^2}\, e^{\sqrt{-1} (k,z) (2 \pi / \ell)}
\end{array}
\end{equation*}
for
$
   \displaystyle
   u = \sum_{k \in \mathbb{Z}^3 \atop k \neq 0} c_k (u) e^{\sqrt{-1} (k,z) (2 \pi / \ell)}.
$
In particular,
 $\mathbf{P}$ is actually the orthogonal projection of $\mathbf{H}^s $ onto
$V_s$ with respect to the unitary structure of $\mathbf{H}^s   $ whenever $s \in {\mathbb Z}_+$.

\begin{lemma}
\label{l.projection.commute}
For each $p \in (1,\infty)$ there is a positive constant $C (p)$ with the property that
\begin{equation}\label{eq.Pi.Lp}
   \| \mathbf{P} u \|_{\mathbf{L}^p   }
 \leq
   C (p)\, \| u \|_{\mathbf{L}^p    }
\end{equation}
for all $u \in \mathbf{L}^p   $.
If moreover $u \in \mathbf{H}^1   $ then
$
   \partial_j (\mathbf{P} u) = \mathbf{P} (\partial_j u)
$
for each $1 \leq j \leq 3$.
\end{lemma}

\begin{proof}
By the very construction, we may identify $\mathbf{P}$ with the $L^2({\mathbb T}^3)\,$-orthogonal
projection and then specify within elliptic pseudodifferential operators of order zero on the
compact closed manifold ${\mathbb T}^3$, see for instance
   \cite[Ch.~I, and Ch.~X1]{Tay81} or
   \cite[\S~2]{Agra90}.

It is well known that the Fourier multipliers (see for instance \cite[Ch. I, and Ch. X1]{Tay81})
are continuous linear self-mappings of $L^p ({\mathbb R}^3)$, if $p \in (0,+\infty)$.
Let $\varPhi$ be the standard fundamental solution of convolution type of the Laplace operator
in $\mathbb{R}^3$.
The operator $\mathrm{rot} \circ \mathrm{rot} \circ \varPhi$ acting on vector fields $u$ with entries
from the Schwartz space $\mathcal{S} ({\mathbb R}^3)$ is actually a matrix Fourier multiplier given by
$
   F^{-1} \left( a (\xi)\, Fu \right)
$
for $u \in \mathbf{L}^2 ({\mathbb R}^3)$, where
   $Fu$ stands for the Fourier transform of $u$,
   by $F^{-1}$ is meant the inverse Fourier transform,
and
   $a (\xi)$ can be identified with the $(3 \times 3)\,$-matrix
$$
   a (\xi) = E_3 - \Big( \frac{\xi_i \xi_j}{|\xi|^2} \Big)
$$
for $\xi \in {\mathbb R}^3 \setminus \{ 0 \}$.
Then, in view of the obvious connection between $\varPhi$ and $\varphi$, on identifying 
periodic functions with functions on the torus ${\mathbb T}^3$ we see that $\mathbf{P}$ is 
a pseudodifferential operator of order zero on the torus. This enables us to apply 
\cite[Ch.~X1, Theorem~2.2]{Tay81} and to conclude that the operator $\mathbf{P}$
maps $\mathbf{L}^p   $ continuously into itself for all $p \in (1,+\infty)$, 
that gives precisely \eqref{eq.Pi.Lp}.

Finally, the commutation relation
$
   \partial_j (\mathbf{P} u) = \mathbf{P} (\partial_j u)
$
is valid since the scalar operator $\partial_j$ commutes with the operators $\mathrm{rot}$ and $\varPhi$
by the very construction.
\end{proof}

We also recall that the Hodge's Theory on the torus, 
see for instance   \cite[\S~2.4]{Agra90},
  \cite[Proposition~1.17]{BerMaj02},
   \cite[\S~2.2]{Tema95}, allows us to introduce bounded linear operators $(- \varDelta)^r$
acting from $H^s   $ into $H^{s-2r}   $ by
$$
   (- \varDelta)^r u
 = -
   \sum_{k \in {\mathbb Z}^3 \atop k \ne 0}
   c_k (u) \left( (k,k) (2 \pi / \ell)^2 \right)^{r} e^{\sqrt{-1} (k,z) (2 \pi / \ell)}
$$
for
$
   \displaystyle
   u = \sum_{k \in {\mathbb Z}^3} c_k (u) e^{\sqrt{-1} (k,z) (2 \pi / \ell)}.
$
If $u \in V_s$ then the action of the operator $\varDelta^r$ reduces to
$$
   (- \varDelta)^r u
 = \sum_{m \in {\mathbb N}_{2,3}}
   \left( m (2 \pi / \ell)^2 \right)^{r} \sum_{j=1}^{J_m} c^{m,j} (u)\, b_{m,j}.
$$
On integrating by parts we obtain
\begin{equation}
\label{eq.by.parts.0}
   \sum_{|\alpha| = j} \| \partial^\alpha u \|^2_{L^2 }
 = \| (- \varDelta)^{j/2} u \|^2_{L^2 }
\end{equation}
for all $u \in H^j $, if $j \in {\mathbb Z}_+$.

\begin{remark}
\label{r.Delta.r}
As all norms on a finite dimensional space are equivalent, there are positive constants $c_1$
and $c_2$ such that
$$
   c_1\, \| (- \varDelta)^{j/2} u \|^2_{L^2   }
 \leq
   \| \nabla^j u \|_{L^2 ({\mathcal Q})}
 \leq
   c_2\, \| (- \varDelta)^{j/2} u \|^2_{L^2   }
$$
for all $u \in H^j   $.
\end{remark}

Thus, in the special case $p = 2$ we may always replace the norm
   $\| \nabla^j u \|_{L^p ({\mathcal Q})}$
with the norm
   $\| (- \varDelta)^{j/2} u \|_{L^2   }$.

\begin{proposition}
\label{c.Sob.d}
For any $s \in {\mathbb Z}$, the differential operator $\nabla$ induces a continuous linear
operator
$$
   \nabla : H^{s+1}  \to \mathbf{H}^{s}  \cap \ker (\mathrm{rot})
$$
which is Fredholm.
Its null-space coincides with the one-dimensional space $\mathbb R$ and,
for any $w \in \mathbf{H}^{s} $, the following are equivalent:

1)
$w = \nabla p$ for some function $p \in H^{s+1} $.

2)
$w \in \dot{H}^{s}  \cap \ker (\mathrm{rot})$.

3)
$\mathbf{P} w = 0$.
\end{proposition}

Hence it follows, in particular, that $\nabla$ establishes an isomorphism between
   $\dot{H}^{s+1} $ and
   $\dot{\mathbf{H}}^{s}  \cap \ker (\mathrm{rot})$.

\begin{proof}
See for instance \cite[Proposition 1.18]{BerMaj02}. 
\end{proof}

We will also use the so-called Bochner spaces of functions of $(x,t)$ in the strip
$\mathbb{R}^3 \times I$, where $I = [0,T]$.
Namely, if
   $\mathcal B$ is a Banach space (possibly, a space of functions on $\mathbb{R}^3$ and
   $p \geq 1$,
we denote by $L^p (I,{\mathcal B})$ the Banach space of all measurable mappings
   $u : I \to {\mathcal B}$
with finite norm
$$
   \| u \|_{L^p (I,{\mathcal B})}
 := \| \|  u (\cdot,t) \|_{\mathcal B} \|_{L^p (I)},
$$
see for instance \cite[Ch.~III, \S~1]{Tema79}.
In the same line stays the space $C (I,{\mathcal B})$, i.e., it is the Banach space of all
mappings $u : I \to {\mathcal B}$ with finite norm
$$
   \| u \|_{C (I,{\mathcal B})}
 := \sup_{t \in I} \| u (\cdot,t) \|_{\mathcal B}.
$$
After Leray \cite{Lera34a,Lera34b}, a great attention was paid to weak solutions to
equations \eqref{eq.NS} in cylinder domains in ${\mathbb R}^3\times [0,\infty)$.
Considering the Navier-Stokes Equations in the Bochner spaces yields the classical existence
theorem for the weak solutions to \eqref{eq.NS}.
To formulate it we set
$$
   \mathbf{D} u = \sum_{j=1}^3 u^j \partial_j u
$$
for a vector field $u = (u^1, u^2, u^3)$.

\begin{theorem}
\label{t.exist.NS.weak}
Given a pair $(f,u_0) \in L^2 (I,V'_1) \times V_0$, there exists a vector field
  $u \in L^\infty (I,V_0) \cap L^2 (I,V_1)$ with
  $\partial_t u \in L^1 (I,V'_1)$, satisfying
\begin{equation}
\label{eq.NS.weak}
\left\{
   \begin{array}{rcl}
   \displaystyle
   \frac{d}{dt} (u,v)_{\mathbf{L}^2   }
 + \mu \sum_{|\alpha| = 1} (\partial^\alpha u, \partial^\alpha v)_{\mathbf{L}^2   }
 & =
 & \langle f - \mathbf{D} u, v \rangle,
\\
   u (\cdot,0)
 & =
 & u_0
   \end{array}
\right.
\end{equation}
for all $v \in V_1$.
\end{theorem}

\begin{proof}
See
   \cite[\S~2.3, \S~2.4]{Tema95}
(or \cite[Ch.~II, Theorem~6.1]{Lion69}
 or \cite[Ch.~III, Theorem~3.1]{Tema79} for domains in ${\mathbb R}^3$
 or the proof of Theorem \ref{t.exist.NS.lin.weak} below).
\end{proof}

Perhaps the space $L^\infty (I,V_0) \cap L^2 (I,V_1)$ is too large in order to achieve even a
uniqueness theorem for \eqref{eq.NS.weak}.
The Bochner spaces $L^{s} (I, L^{r} ({\mathbb R}^3))$ with
\begin{equation}
\label{eq.s.r}
\begin{array}{ccc}
   \displaystyle \frac{2}{\mathfrak{s}} + \frac{3}{\mathfrak{r}} = 1,
 & 2 \leq \mathfrak{s}  < \infty,
 & 3 < \mathfrak{r} \leq \infty
\end{array}
\end{equation}
are well known to be uniqueness and regularity classes for the Navier-Stokes equations, see
   \cite{Pro59},
   \cite{Lady67},
   \cite{Serr62}.
The limit case
   $\mathfrak{s} =\infty$,
   $\mathfrak{r} = 3$
was added to the list in \cite{ESS03} but we will not discuss it here.

\begin{theorem}
\label{t.NS.unique}
Let $\mathfrak{s}$ and $\mathfrak{r}$ satisfy \eqref{eq.s.r}.
For each data
  $(f,u_0) \in L^2 (I,V'_1) \times V_0$,
the nonlinear Navier-Stokes equations 
\eqref{eq.NS.weak} possess at most one solution in the
space
   $L^\infty (I,V_0) \cap L^2 (I,V_1) \cap L^\mathfrak{s} (I,\mathbf{L}^\mathfrak{r}   )$.
\end{theorem}

\begin{proof}
It is similar to the proof for the flows in domains of ${\mathbb R}^3$, see
   \cite[Ch.~6, \S~2, Theorem~1]{Lady70},
   \cite[Theorem 6.9]{Lion69} or Theorem 3.4 and Remark 3.6 in \cite{Tema79}
or Lemma \ref{l.NS.lin.unique} below.
\end{proof}

We now proceed with studying more regular solutions.

\section{An open mapping theorem}
\label{s.OM}

This section is devoted to the so-called stability property for solutions to the 
Navier-Stokes equations. One of the first statements of this kind was obtained  by 
O.A.Ladyzhenskaya \cite[Ch.~4, \S~4, Theorem~11]{Lady70} 
for flows in bounded domains in ${\mathbb R}^3$ with $C^2$ smooth boundaries.

In order to extend the property to the spaces of high smoothness, we consider the standard linearisation of problem \eqref{eq.NS} at the zero solution $(0,0)$.
Namely, given sufficiently regular spatially periodic functions
   $f = (f^1, f^2, f^3)$,
   $w = (w^1, w^2 , w^3)$
on ${\mathbb R}^3 \times [0,T]$ and
   $u_0 = (u^1_{0}, u^2_{0}, u^3_{0})$
on ${\mathbb R}^3$ with values in $\mathbb{R}^3$, find sufficiently regular spatially 
periodic functions
   $u = (u^1, u^2, u^3)$ and
   $p$
in the strip ${\mathbb R}^3 \times [0,T]$ which satisfy
\begin{equation}
\label{eq.NS.lin}
\left\{
\begin{array}{rcll}
   \partial _t u   -  \mu \varDelta u + (w \cdot \nabla) u  + (u \cdot \nabla) w + \nabla p
 & =
 & f,
 & (x,t) \in {\mathbb R}^3 \times (0,T),
\\
   \mbox{div}\, u
 & =
 & 0,
 & (x,t) \in {\mathbb R}^3 \times (0,T),
\\[.05cm]
   u
& =
& u_0,
& (x,t) \in \mathbb{R}^3 \times \{ 0 \}
\end{array}
\right.
\end{equation}

Considering this problem in the Bochner spaces yields the classical existence theorem for the 
weak solutions to \eqref{eq.NS.lin}.
To formulate it we set
$$
   \mathbf{B} (w,u)
 = \sum_{j=1}^3 w^j \partial_j u + \sum_{j=1}^3 u^j \partial_j w
$$
for vector fields
   $u = (u^1, u^2, u^3)$ and
   $w = (w^1, w^2, w^3)$.

\begin{theorem}
\label{t.exist.NS.lin.weak}
Suppose $w \in C (I,V_0) \cap L^2 (I,V_1) \cap L^{2} (I,\mathbf{L}^{\infty}   )$.
Given any pair $(f,u_0) \in L^2 (I,V'_1) \times V_0$, there is a unique vector field
   $u \in C (I,V_0) \cap L^2 (I,V_1)$ with
   $\partial_t u \in L^2 (I,V'_1)$,
satisfying
\begin{equation}
\label{eq.NS.lin.weak}
\left\{
   \begin{array}{rcl}
   \displaystyle
   \frac{d}{dt} (u,v)_{\mathbf{L}^2   }
 + \mu \sum_{|\alpha| = 1} (\partial^\alpha u, \partial^\alpha v)_{\mathbf{L}^2   }
 & =
 & \langle f - \mathbf{B} (w,u), v \rangle,
\\
   u (\cdot,0)
 & =
 & u_0
   \end{array}
\right.
\end{equation}
for all $v \in V_1$.
\end{theorem}

\begin{proof}
It is similar to the proof of the uniqueness and existence theorem for the Stokes problem and
the Navier-Stokes problem, see
   \cite[\S~2.3, \S~2.4]{Tema95}
(or
   \cite[Ch.~II, Theorem 6.1 and Theorem 6.9]{Lion69} or
   \cite[Ch.~III, Theorem 1.1, Theorem 3.1 and Theorem 3.4]{Tema79}
for domains in ${\mathbb R}^3$).
We shortly recall the arguments in the part we will use in order to obtain existence theorems
related to \eqref{eq.NS.lin.weak} for more regular data and solutions.

We first note that, for $s \in {\mathbb Z}_+$, the system
$$
   \{ e_1, e_2, e_3, v_{m,j} \}_{m \in {\mathbb N}_{2,3} \atop j = 1, \ldots, J_m}
$$
is an orthogonal basis in $V_s$, see Proposition \ref{p.Vs}.
It is convenient to set
   $J_0 = 3$ and
   $v_{0,j} = e_j$
for $j = 1, 2, 3$.
Next, one defines the Faedo-Galerkin approximations in the usual way
$$
   u_M
 = \sum_{m \in \{ 0 \} \cup {\mathbb N}_{2,3} \atop 0 \leq m \leq M}
   \sum_{j=1}^{J_{m}} c^{m,j}_M (t) v_{m,j} (x)
$$
where the functions $c^{m,j}_M$ satisfy the relations
\begin{equation}
\label{eq.ODU.lin.0}
\begin{array}{rcl}
   \displaystyle
   (\partial_t u_M\!, v_{m,j})_{\mathbf{L}^2   }
   \!
 + \mu \sum_{i=1}^3 (\partial_i u_M\!, \partial_i v_{m,j})_{\mathbf{L}^2   }
   \!
 + (\mathbf{B} (w,u_M\!), v_{m,j})_{\mathbf{L}^2   }
 \!\!
 & \!\! = \!\!
 \!\!
 & \langle f, v_{m,j} \rangle ,
\\
   u_M (x,0)
 \!\!
 & \!\! = \!\!
 \!\!
 & u_{0,M} (x)
\end{array}
\end{equation}
for all
   $m \in \{ 0 \} \cup \{ 0 \} \cup \mathbb{N}_{2,3}$ with $0 \leq m \leq M$
and
   $0 \leq j \leq J_m$.
The initial datum $u_{0,M}$ is the orthogonal projection of $u_0$ into the linear span of the
system
$
   \{ v_{m,j} \},
$
where
   $m$ varies over the set $(\{ 0 \} \cup {\mathbb N}_{2,3}) \cap \{ 0, 1, \ldots, M \}$
and
   $j$ from $1$ to $J_m$.
In this way \eqref{eq.ODU.lin.0} reduces to an initial problem for the unknown coefficients
$c^{m,j}_M (t)$ on the interval $[0,T]$.
That is
\begin{equation}
\label{eq.ODU.lin.1}
\left\{
\begin{array}{rcl}
   \displaystyle
   \frac{d}{dt} c^{m,j}_M
 + \mu m \Big( \frac{2 \pi}{\ell} \Big)^2 c^{m,j}_M
 + \sum_{m' \in \{ 0 \} \cup {\mathbb N}_{2,3} \atop 0 \leq m' \leq M}
   \sum_{j'=1}^{J_{m'}}
   w^{m,j}_{m',j'} c^{m',j'}_M
 & =
 & f^{m,j},
\\[-.2cm]
   c^{m,j}_M (0)
 & =
 & u_0^{m,j},
\end{array}
\right.
\end{equation}
where the real-valued function $f^{m,j} (t) = \langle f (\cdot,t), v_{m,j} \rangle$ belong 
to the class $L^2 [0,T]$, $u_{0}^{m,j} = (u_{0}, v_{m,j})_{\mathbf{L}^2   }$ are real 
numbers, and
\begin{equation}
\label{eq.C}
\begin{array}{rcl}
   w^{m,j}_{m',j'} (t)
 & =
 & (w (\cdot,t) \cdot \nabla v_{m',j'}, v_{m,j})_{\mathbf{L}^2   }
 + (v_{m',j'} \cdot \nabla w (\cdot,t), v_{m,j})_{\mathbf{L}^2   }
 \\
 & =
 & (w (\cdot,t) \cdot \nabla v_{m',j'}, v_{m,j})_{\mathbf{L}^2   }
 - (w (\cdot,t), v_{m',j'} \cdot \nabla v_{m,j})_{\mathbf{L}^2   }
 \end{array}
\end{equation}
belong to $L^\infty [0,T]$.

Denote by
   $c_M$
and
   $f_M$,
   $u_{0,M}$
the vectors constructed from the components
   $c^{m,j}_M$
and
   $f^{m,j}$,
   $u_0^{m,j}$,
where
   $m \in \{ 0 \} \cup \mathbb{N}_{2,3}$ does not exceed $M$ and
   $j = 1, \ldots, J_m$,
respectively, ordered in the lexicographic order.
Let moreover $A_M (t)$ stand for the corresponding matrix of real-valued functions on $[0,T]$
constructed from the components
\begin{equation}
\label{eq.CC}
   a^{m,j}_{m',j'} (t)
 = \mu m \Big( \frac{2 \pi}{\ell} \Big)^2 \delta^{m,j}_{m',j'}
 + w^{m,j}_{m',j'} (t),
\end{equation}
where $\delta^{m',j}_{m',j'}$ is the Kronecker symbol.
Then \eqref{eq.ODU.lin.1} transforms to
$$
\left\{
\begin{array}{rcl}
   \displaystyle
   \frac{d}{dt}\, c_M (t) + A_M (t) c_M (t)
 & =
 & f_M (t),
\\
   c_M (0)
 & =
 & u_{0,M},
\end{array}
\right.
$$
and hence for each $M \in \mathbb N$ the system \eqref{eq.ODU.lin.0} admits a unique solution
$c_M$ on the interval $[0,T]$ given by
$$
   c_M (t)
 = \exp \Big( - \int_0^t A_M (s) ds \Big) u_{0,M}
 + \int_0^t \exp \Big( \int_0^s A_M (s') d s' \Big) f_M (s) ds
$$
where
$$
   \exp \Big( - \int_0^t A_M (s) ds \Big)
 = \sum _{k=0}^\infty \frac{1}{k!} \Big( - \int_0^t A_M (s) ds \Big)^k.
$$
Formula \eqref{eq.C} shows that the components
   $a^{m,j}_{m',j'} (t)$,
and so the entries of the matrix
$$
   \displaystyle \exp \Big( - \int_0^t A_M (s) ds \Big),
$$
belong actually to $C^{0,1} [0,T]$.
Hence, as $f^{m,j} \in L^2 [0,T]$,
   the components of the vector $c_M$ are of class $C^{1/2} [0,T]$
and
   the components of the vector $(d/dt)\, c_M$ belong to $L^2 [0,T]$.
In particular,
   $u_M \in L^2 (I,V_s)\cap C (I,V_s)$
for each $s \in {\mathbb Z}_+$, and
   $(d/dt) u_M \in L^2 (I,V'_1)$.

In order to obtain a solution to \eqref{eq.NS.lin} one usually appeals to a priori estimates.
To obtain them, we invoke the following useful lemma by J.-L. Lions.

\begin{lemma}
\label{l.Lions}
Let $V$, $H$ and $V'$ be Hilbert spaces such that $V'$ is the dual to $V$ and the embeddings
$
   V \subset H \subset V'
$
are continuous and everywhere dense.
If
   $u \in L^2 (I,V_1)$ and
   $\partial_t u \in L^2 (I,V'_1)$
then
\begin{equation}
\label{eq.dt}
   \frac{d}{dt} \| u (\cdot, t) \|^2_{H}
 = 2\, \langle \partial_t u, u \rangle
\end{equation}
and $u$ is equal almost everywhere to a continuous mapping from $[0,T]$ to $H$.
\end{lemma}

\begin{proof}
See \cite[Ch.~III, \S~1, Lemma~1.2]{Tema79}.
\end{proof}

Next we note that if spatially periodic vector fields $u$, $v$ and $w$ are sufficiently regular
and $w$ satisfies $\mathrm{div}\, w = 0$ in ${\mathbb R}^3$ then
\begin{equation}
\label{eq.Bu}
   (w \cdot \nabla u, v)_{\mathbf{L}^2   }
 = \sum_{i, j=1}^3 ( w_j \partial_j u_i, v_i)_{L^2   }
 = - \sum_{i, j=1}^3 (u_i, w_j \partial_j  v_i )_{L^2   }
\end{equation}
(see for instance \cite[Ch.3, \S~3.1, formula~(3.2)]{Tema79}).
In particular, we deduce from \eqref{eq.Bu} that
\begin{equation}
\label{eq.Bu.zero}
\begin{array}{rcl}
   (w \cdot \nabla u, u)_{\mathbf{L}^2   }
 & =
 & 0,
\\
   (u \cdot \nabla w, u)_{\mathbf{L}^2   }
 & =
 & - (w, u \cdot \nabla u)_{\mathbf{L}^2   }
\end{array}
\end{equation}
for all sufficiently regular divergence free vector fields $u$ and $w$.
Thus, on multiplying the equation corresponding to indices $m$ and $j$ in \eqref{eq.ODU.lin.0}
by $c^{m,j}_M$, and summing up with respect to $m$ and $j$ we get
\begin{equation}
\label{eq.um.bound}
   \frac{1}{2} \frac{d}{dt}\, \| u_M (\cdot,t) \|^2_{\mathbf{L}^2   }
 + \mu\, \| \nabla u_M \|^2_{\mathbf{L}^2   }
 =  \langle f, u_M \rangle
 + (w, u_M \cdot \nabla u_M)_{\mathbf{L}^2   }
\end{equation}
for all $t \in [0,T]$ because of \eqref{eq.dt} and \eqref{eq.Bu.zero}.

The following standard statement, where
\begin{eqnarray*}
   \| u \|_{k,\mu,T}
 & =
 & \Big( \| \nabla^k u \|^2_{C (I, \mathbf{L}^2   )}
       + \mu\, \| \nabla^{k+1} u \|^2_{L^2 (I, \mathbf{L}^2   )}
   \Big)^{1/2},
\\
   \| (f,u_0) \|_{0,\mu,T}
 & =
 & \Big( \| u_0 \|^2_{\mathbf{L}^2   }
       + \frac{2}{\mu}\, \| f \|^2_{L^2 (I, V_1')}
       + \| f \|^2_{L^1 (I, V_1')}
   \Big)^{1/2},
\end{eqnarray*}
gives a basic a priori estimate for  regular solutions to \eqref{eq.NS.lin.weak} 
and \eqref{eq.NS.weak}.

\begin{lemma}
\label{p.En.Est.u}
Let
   $w \in L^2 (I,V_1) \cap C (I,V_0) \cap L^{2} (I, \mathbf{L}^{\infty}   )$.
If
   $u \in C (I,V_0) \cap L^2 (I,V_1)$ and
   $(f,u_0) \in L^2 (I,V_1') \times V_0$
satisfy
\begin{equation}
\label{eq.En.equal}
\left\{
\begin{array}{rcl}
   \displaystyle
   \frac{1}{2} \frac{d}{d\tau}\, \| u (\cdot,\tau) \|^2_{\mathbf{L}^2   }
 + \mu\, \| \nabla u \|^2_{\mathbf{L}^2   }
 & =
 & \langle f, u \rangle
 + (w, u \cdot \nabla u)_{\mathbf{L}^2   },
\\
   u (\cdot,0)
 & =
 & u_0
\end{array}
\right.
\end{equation}
for all $t \in [0,T]$, then
\begin{equation}
\label{eq.En.Est2}
\begin{array}{rcl}
   \displaystyle
   \| u \|^2_{0,\mu,T}
 & \leq &
   \displaystyle
   \| (f,u_0) \|^2_{0,\mu,T}
   \Big(
     1
   + 2 \sqrt{2}
     \exp \Big( \frac{1}{\mu} \int_0^T \| w (\cdot,t) \|^{2}_{\mathbf{L}^{\infty}   } dt \Big)
\\
 & + &
   \displaystyle
   \frac{4}{\mu}
   \Big( \int_0^T \| w (\cdot,t) \|^{2}_{\mathbf{L}^{\infty}   } dt \Big)
   \exp \Big( \frac{2}{\mu} \int_0^T \| w (\cdot,t)\|^{2}_{\mathbf{L}^{\infty}   } dt \Big)
 \Big) .
\end{array}
\end{equation}
\end{lemma}

It is easy to see that
\begin{equation}
\label{eq.En.Est2add}
   \| u \|_{L^{p} (I, \mathbf{L}^2   )}
 \leq T^{1/p}\, \| u \|_{L^\infty (I,\mathbf{L}^2   )}
\end{equation}
holds for any $p \geq 1$, which accomplishes the energy estimate \eqref{eq.En.Est2}.

\begin{proof}
It is similar to the proof of energy estimates for solutions to the Navier-Stokes equations, see
   \cite[Ch.~IV, \S~3]{Lady70} or
   \cite[Ch.~III, Theorem~3.1]{Tema79} or
   \cite[Remark~3.4]{Tema95}
for the periodic case.

The H\"older inequality and \eqref{eq.Young} imply
\begin{eqnarray}
\label{eq.w.infty}
   2 \left| \int_0^t (w, u \cdot \nabla u)_{\mathbf{L}^2   } ds \right|
 & \leq &
   2 \int_0^t
   \| \nabla u \|_{\mathbf{L}^2   }
   \| w \|_{\mathbf{L}^\infty   }
   \| u \|_{\mathbf{L}^{2}   }
   ds
\nonumber
\\
 & \leq &
   \frac{\mu}{2} \int_0^t \| \nabla u \|^2_{\mathbf{L}^2   } ds
 + \frac{2}{\mu} \int_0^t
   \| w \|^2_{\mathbf{L}^\infty   }
   \| u \|^2_{\mathbf{L}^{2}   }
   ds.
\nonumber
\\
\end{eqnarray}
On the other hand, by \eqref{eq.v.dual}, we get
\begin{eqnarray}
\label{eq.En.Est10}
\lefteqn{
   2\, \left| \int_0^t \langle f (\cdot,s), u (\cdot,s) \rangle ds \right|
}
\nonumber
\\
 & \leq &
   2 \int_0^t
   \| f (\cdot,s) \|_{V_1'}
   \Big( \| \nabla u (\cdot,s) \|^2_{\mathbf{L}^{2}   }
       + \| u (\cdot,s) \|^2_{\mathbf{L}^{2}   }
   \Big)^{1/2}
   ds
\nonumber
\\
 & \leq &
   2 \int_0^t
   \| f (\cdot,s) \|_{V_1'}
   \Big( \| \nabla u (\cdot,s) \|_{\mathbf{L}^{2}   }
       + \| u (\cdot,s) \|_{\mathbf{L}^{2}   }
   \Big)
   ds
\nonumber
\\
 & \leq &
   \int_0^t
   \Big( \frac{2}{\mu} \| f (\cdot,s) \|^2_{V_1'}
       + \frac{\mu}{2} \| \nabla u (\cdot,s) \|^2_{\mathbf{L}^{2}   }
       + 2 \| f (\cdot,s) \|_{V_1'} \| u (\cdot,s) \|_{\mathbf{L}^{2}   }
   \Big)
   ds
\nonumber
\\
\end{eqnarray}
for all $t \in [0,T]$.
Integrating \eqref{eq.En.equal} with respect to 
$\tau$ over $[0,t]$ and taking both \eqref{eq.w.infty}
and \eqref{eq.En.Est10} into account yields
\begin{eqnarray}
\label{eq.En.Est1}
\lefteqn{
   \| u (\cdot, t) \|^2_{\mathbf{L}^2   }
 + \mu \int_0^t \| \nabla u (\cdot,s) \|^2_{\mathbf{L}^2   } ds
 \, \leq \,
   \| u_0 \|^2_{\mathbf{L}^2   }
}
\nonumber
\\
 & + &
   \int_0^t
   \Big( \frac{2}{\mu} \| f (\cdot,s) \|^2_{V_1'}
       + 2 \| f (\cdot,s)\|_{V_1'} \| u (\cdot,s) \|_{\mathbf{L}^{2}   }
       + \frac{2}{\mu} \| w (\cdot,s) \|^{2}_{\mathbf{L}^{\infty}   }
                       \| u (\cdot,s) \|^2_{\mathbf{L}^{2}   }
   \Big)
   ds.
\nonumber
\\
\end{eqnarray}

Finally, on applying Lemma \ref{l.Perov} with
   $\gamma=1/2$ and
   $Y (t) = \| u (\cdot,t) \|^2_{\mathbf{L}^{2}   }$
we readily obtain
\begin{eqnarray*}
   \| u (\cdot,t) \|^2_{\mathbf{L}^{2}   }
 & \leq &
   \Big(
   \Big( \| u_0 \|^2_{\mathbf{L}^2   } + \frac{2}{\mu} \| f \|^2_{L^2 ([0,t],V_1')}
   \Big)^{1/2}
   \exp \Big( \frac{1}{\mu} \int_0^t \| w (\cdot,s) \|^{2}_{\mathbf{L}^{\infty}   } ds \Big)
\\
 & + &
   \int_0^t
   \| f (\cdot,s) \|_{V_1'}
   \exp
   \Big( \frac{1}{\mu} \int_{s}^t \| w (\cdot,s') \|^{2}_{\mathbf{L}^{\infty}   } d s'
   \Big)
   ds
   \Big)^{2}
\\
 & \leq &
   2\, \| (f,u_0) \|^2_{0,\mu,T}\,
   \exp \Big( \frac{2}{\mu} \int_0^T  \| w (\cdot,s) \|^{2}_{\mathbf{L}^{\infty}   } ds \Big)
\end{eqnarray*}
for all $t \in [0,T]$.
Estimate \eqref{eq.En.Est2} follows from the latter inequality.
\end{proof}

Lemma \ref{p.En.Est.u} and \eqref{eq.um.bound} imply that the sequence $\{ u_M \}$ is bounded
in the space $C (I,V_0) \cap L^2 (I,V_1)$.
So, it bounded in $L^\infty (I,V_0) \cap L^2 (I,V_1)$ and we can extract a subsequence that
   converges weakly-$^\ast$ in $L^\infty (I,V_0)$ and
   converges weakly in $L^2 (I,V_1)$
to an element $u \in L^\infty (I,V_0) \cap L^2 (I,V_1)$.
For abuse of notation, we use the same designation $\{ u_M \}$ for such a subsequence.

At this point, rather delicate arguments involving compact embedding theorems for the 
Bochner-Sobolev spaces on bounded domains show that the sequence $\{ u_M \}$ may be 
considered as convergent in the space $L^2 (I,\mathbf{L}^2   )$, see
   \cite[Ch.~II, Theorem~6.1]{Lion69} 
or \cite[Ch.~III, Theorem~3.1]{Tema79}.
This allows us to pass to the limit \hfill with respect to 
$M \to \infty$ in \eqref{eq.ODU.lin.0} and
to conclude that the element $u$ satisfies \eqref{eq.NS.lin.weak}.

We proceed with the uniqueness.

\begin{lemma}
\label{l.NS.lin.unique}
Let
   $w \in  L^2 (I,V_1) \cap L^\infty (I,V_0) \cap L^{2} (I,\mathbf{L}^{\infty}   )$.
For each pair
  $(f,u_0) \in L^2 (I,V'_1) \times V_0$
the linearised Navier-Stokes equations \eqref{eq.NS.lin.weak} have at most one solution in the
space
  $L^2 (I,V_1) \cap L^\infty (I,V_0)$.
\end{lemma}

\begin{proof}
First, we note that \eqref{eq.Bu} implies
$$
   (\mathbf{B} (w,u), v)_{\mathbf{L}^2   }
 = - (w, u \cdot \nabla v)_{\mathbf{L}^2   }
   - (u, w \cdot \nabla v)_{\mathbf{L}^2   }
$$
for all $u, v \in  L^2 (I,V_1) \cap L^\infty (I,V_0)$.
Hence, if
 $w \in  L^2 (I,V_1) \cap L^\infty (I,V_0) \cap L^{2} (I,\mathbf{L}^{\infty}    )$
then the H\"older inequality yields
$$
   |(\mathbf{B} (w,u), v)_{\mathbf{L}^2   }|
 \leq
   2\,
   \| w \|_{\mathbf{L}^{\infty}   }
   \| u \|_{\mathbf{L}^{2}   }
   \| \nabla v \|_{\mathbf{L}^2   }
$$
for all $v \in V_1$.
On applying the H\"older inequality once again we readily conclude that
$$
   \|(\mathbf{B} (w,u) \|^2_{L^2 (V'_1)}
 \leq
   \int_0^T \| w \|^2_{\mathbf{L}^{\infty}   }
            \| u \|^2_{\mathbf{L}^{2}   }\,
   dt
 \leq
   \| w \|^{2}_{L^{2} (I,\mathbf{L}^{\infty}    )}
   \| u \|^{2}_{L^{\infty} (I,\mathbf{L}^2   )},
$$
i.e.,
   $\mathbf{B} (w, u) \in L^2 (I,V'_1)$ and
   $\partial_t u \in L^2 (I,V'_1)$,
if $u$ is a solution to problem \eqref{eq.NS.lin.weak}.

Let now $u'$ and $u''$ be any two solutions to \eqref{eq.NS.lin.weak} from the declared function
space.
Then the difference $u = u' - u''$ is a solution to \eqref{eq.NS.lin.weak} with zero data
   $(f,u_0) = (0,0)$.
Hence it follows that
$$
   \langle \partial_t u, u \rangle +  \| \nabla u \|^2_{\mathbf{L}^2   }
 = \langle \mathbf{B} (w,u), u \rangle.
$$
Next, as
   $u \in L^2 (I,V_1)$ and
   $\partial_t u \in L^2 (I,V'_1)$,
integrating the above equality with respect to $t$ and using  Lemma \ref{l.Lions}, we get
$$
   \| u (\cdot,t) \|^2_{\mathbf{L}^2   }
 + 2 \mu  \int_0^t \| \nabla u (\cdot,s) \|^2_{\mathbf{L}^2   } ds
 = 2 \int_0^t \langle \mathbf{B} (w,u), u \rangle\, ds
$$
because
   $\mathbf{B} (w,u) \in L^2 (I,V'_1)$
(and $u \in L^2 (I,V_1)$).
As the vector field $w$ is assumed to belong to
   $L^2 (I,V_1) \cap L^\infty (I,V_0) \cap L^{2} (I,\mathbf{L}^{\infty}   )$,
using \eqref{eq.Bu.zero} and \eqref{eq.En.Est1} gives
\begin{equation}
\label{eq.unique.est}
   \| u (\cdot,t) \|^2_{\mathbf{L}^2   }
\leq
  \frac{2}{\mu}
  \int_0^t \| w (\cdot,s) \|^{2}_{\mathbf{L}^{\infty}   }
           \| u (\cdot,s) \|^2_{\mathbf{L}^2    }\, ds.
\end{equation}
Applying Gronwall's Lemma \ref{l.Perov} to this inequality yields
$$
   0
 \leq
   \| u (\cdot,t) \|^2_{\mathbf{L}^2}
 \leq
   0
$$
for all $t \in [0,T]$, and so $u \equiv 0$, as desired.
\end{proof}

Finally, the vector field $u$ belongs to $C (I,V_0)$, for
   $\partial_t u \in L^2 (I,V'_1)$,
   $u \in L^2 (I,V_1)$
and the embeddings $V_1 \subset V_0 \subset V_1'$ are continuous and everywhere dense, i.e.,
the assumptions of Lemma \ref{l.Lions} are fulfilled.
\end{proof}

Of course, Theorem \ref{t.exist.NS.lin.weak} can be easily extended to the case where
   $w \in L^\mathfrak{s} (I, \mathbf{L}^\mathfrak{r}   )$
with indices
   $\mathfrak{s}$ and
   $\mathfrak{r}$
satisfying \eqref{eq.s.r}, see for instance
   \cite{Lady67},
   \cite[Ch.~3, \S~1-\S~4]{LadSoUr67}
and elsewhere.

We are now in a position to introduce appropriate function spaces for solutions and for the 
data in order to obtain existence theorems for regular solutions to the Navier-Stokes 
equations. More precisely, as the principal differential part of the Navier-Stokes equations 
is parabolic, we prefer to follow the dilation principle 
when  introducing function spaces for 
the unknown velocity and given exterior forces. Namely, for $s, k \in {\mathbb Z}_+$, we 
denote by   $B^{k,2s,s}_\mathrm{vel} (I)$
the set of all vector fields $u$ in
   $C (I,V_{k+2s}) \cap L^2 (I,V_{k+1+2s})$
such that
$$
   \partial_x^\alpha \partial_t^j u
 \in
   C (I,V_{k+2s-|\alpha|-2j}) \cap L^2 (I,V_{k+1+2s-|\alpha|-2j})
$$
provided $|\alpha|+2j \leq 2s$.
We endow each space $B^{k,2s,s}_\mathrm{vel} (I)$ with the natural norm
$$
   \| u \|_{B^{k,2s,s}_\mathrm{vel} (I)}
 :=
   \Big( \sum_{i=0}^k \sum_{|\alpha|+2j \leq 2s}
         \| \partial_x^\alpha \partial_t^j u \|^2_{i,\mu,T}
   \Big)^{1/2}
$$
where
$
   \displaystyle
   \| u \|_{i,\mu,T}
 = \Big( \| \nabla^i u \|^2_{C (I,\mathbf{L}^2   )}
       + \mu \| \nabla^{i+1} u \|^2_{L^2 (I,\mathbf{L}^2   )}
   \Big)^{1/2}
$
are seminorms on the space $B^{k,2s,s}_\mathrm{vel} (I)$, too. 

Similarly, for $s, k \in {\mathbb Z}_+$, we define the space $B^{k,2s,s}_\mathrm{for} (I)$
to consist of all forces $f$ in
$
   C (I,\mathbf{H}^{2s+k}   ) \cap L^2 (I,\mathbf{H}
	^{2s+k+1}   )
$
with the property that
$
   \partial_x^\alpha \partial _t^j f
 \in
   C (I,\mathbf{H}^{k}   ) \cap L^2 (I,\mathbf{H}^{k+1}   )
$
provided
$
   |\alpha|+2j \leq 2s.
$
If $f \in B^{k,2s,s}_\mathrm{for} (I)$, then actually
$$
   \partial_x^\alpha \partial _t^j f
 \in
   C (I,\mathbf{H}^{k+2(s-j)-|\alpha|}   )
 \cap
   L^2 (I,\mathbf{H}^{k+1+2(s-j)-|\alpha|}   )
$$
for all $\alpha$ and $j$ satisfying $|\alpha|+2j \leq 2s$.
We endow each space $B^{k,2s,s}_\mathrm{for} (I)$ with the natural norm
$$
   \| f \|_{B^{k,2s,s}_\mathrm{for} (I)}
 :=
   \Big( \sum_{i=0}^k \sum_{|\alpha|+2j \leq 2s}
         \| \nabla^i \partial_x^\alpha \partial_t^j f \|^2_{C (I,\mathbf{L}^2({\mathcal Q}))}
       + \| \nabla^{i+1} \partial_x^\alpha \partial_t^j f \|^2_{L^2 (I,\mathbf{L}^2 
			({\mathcal Q}))}  \Big)^{1/2}.
$$
Finally, the target space for the pressure $p$ is $B^{k+1,2s,s}_\mathrm{pre} (I)$.
By definition, it consists of all functions $p \in C (I,\dot{H}^{2s+k+1}   ) \cap L^2 
(I,\dot{H}^{2s+k+2}   )$ such that $\nabla p \in B^{k,2s,s}_\mathrm{for} (I)$. 
Obviously, the space does not contain functions depending on $t$ only, and this allows us to 
equip it with the norm
$$
   \| p \|_{B^{k+1,2s,s}_\mathrm{pre} (I)}
 = \| \nabla p \|_{B^{k,2s,s}_\mathrm{for} (I)}.
$$

It is easy to see that
   $B^{k,2s,s}_\mathrm{vel} (I)$,
   $B^{k,2s,s}_\mathrm{for} (I)$ and
   $B^{k+1,2s,s}_\mathrm{pre} (I)$
are Banach spaces.
We proceed with a simple lemma.

\begin{lemma}
\label{l.NS.cont}
Suppose that
   $s \in \mathbb N$,
   $k \in {\mathbb Z}_+$ and
   $w \in B^{k,2s,s}_\mathrm{vel} (I)$.
The following mappings are continuous:
$$
\begin{array}{rrcl}
   \nabla :
 & B^{k+1,2(s-1),s-1}_\mathrm{pre} (I)
 & \to
 & B^{k,2(s-1),s-1}_\mathrm{for} (I),
\\[.2cm]
   \varDelta :
 & B^{k,2s,s}_\mathrm{vel} (I)
 & \to
 & B^{k,2(s-1),s-1}_\mathrm{for} (I),
\\
   \varDelta :
 & B^{k+2,2(s-1),s-1}_\mathrm{vel} (I)
 & \to
 & B^{k,2(s-1),s-1}_\mathrm{for} (I),
\\[.2cm]
   \partial_t :
 & B^{k,2s,s}_\mathrm{vel} (I)
 & \to
 & B^{k,2(s-1),s-1}_\mathrm{for} (I),
\\
   \delta_t :
 & B^{k,2s,s}_\mathrm{vel} (I)
 & \to
 & V_{k+2s},
\\[.2cm]
   \mathbf{B} (w,\cdot) :
 & B^{k+2,2(s-1),s-1}_\mathrm{vel} (I)
 & \to
 & B^{k,2(s-1),s-1}_\mathrm{for} (I),
\\
  \mathbf{B} (w,\cdot) :
 & B^{k,2s,s}_\mathrm{vel} (I)
 & \to
 & B^{k,2(s-1),s-1}_\mathrm{for} (I),
\\[.2cm]
   \mathbf{D} :
 & B^{k+2,2(s-1),s-1}_\mathrm{vel} (I)
 & \to
 & B^{k,2(s-1),s-1}_\mathrm{for} (I),
\\
   \mathbf{D} :
 & B^{k,2s,s}_\mathrm{vel} (I)
 & \to
 & B^{k,2(s-1),s-1}_\mathrm{for} (I).
\end{array}
$$
\end{lemma}

As usual, we write $\delta_t (u (\cdot,t)) = u (\cdot,0)$ for the initial value functional
(or the delta-function in $t$).

\begin{proof}
Indeed, the linear operators in question are continuous by the very definition of the function
spaces.

By definition, the space $B^{2,0,0}_\mathrm{vel} (I) $ is continuously embedded into the spaces
   $C (I,\mathbf{H}^2   )$ and
   $L^2 (I,\mathbf{H}^3    )$.
Then, by the Sobolev embedding theorem (see for instance \cite[Ch.~4, Theorem~4.12]{Ad03}),
for any
   $k, s \in {\mathbb Z}_+$ and
   $\lambda\in (0,1)$
satisfying
\begin{equation}
\label{eq.Sob.index}
   k - s - \lambda > 3/2,
\end{equation}
there exists a constant $c (k,s,\lambda)$ depending on the parameters, such that
\begin{equation}
\label{eq.Sob.indexc}
   \| u \|_{C^{s,\lambda}    }
 \leq c (k,s,\lambda)\, \| u \|_{H^{k}   }
\end{equation}
for all $u \in H^{k}   $.
Hence, $B^{2,0,0}_\mathrm{vel} (I) $ is continuously embedded into
   $C (I,\mathbf{L}^{\infty}   )$ and
   $L^2 (I,\mathbf{W}^{1,\infty}   )$.
Then
\begin{eqnarray}
\label{eq.B.cont.1}
   \| \mathbf{B} (w,u) \|^2_{\mathbf{L}^2   }
 & \leq &
   \| w \|^2_{\mathbf{L}^\infty   } \| \nabla u \|^2_{\mathbf{L}^2   }
 + \| \nabla w \|^2_{\mathbf{L}^2   } \| u \|^2_{\mathbf{L}^\infty   }
\nonumber
\\
 & \leq &
   c
   \Big( \| w \|^2_{\mathbf{H}^2   } \| \nabla u \|^2_{\mathbf{L}^2    }
       + \| \nabla w \|^2_{\mathbf{L}^2    } \| u \|^2_{\mathbf{H}^2   }
   \Big),
\nonumber
\\
\end{eqnarray}
the constant $c$ being independent of $u$ and $w$, and so
$$
   \|\mathbf{B} (w,u) \|^2_{C (I,\mathbf{L}^2   )}
 \leq
   c
   \Big( \| w \|^2_{C (I,\mathbf{H}^2   )} \| \nabla u \|^2_{C (I,\mathbf{L}^2   )}
       + \| \nabla w \|^2_{C (I,\mathbf{L}^2   )} \| u \|^2_{C (I,\mathbf{H}^2   )}
   \Big).
$$
As
\begin{equation}
\label{eq.Leib}
   \partial_j \mathbf{B} (w,u)
 = \mathbf{B} (\partial_j w,u) + \mathbf{B} (w,\partial_j u),
\end{equation}
we also get
\begin{eqnarray}
\label{eq.B.cont.3}
\lefteqn{
   \| \nabla \mathbf{B} (w,u) \|^2_{\mathbf{L}^2   }
}
\nonumber
\\
 & \leq &
   c
   \Big( \| w \|^2_{\mathbf{L}^\infty   } \| \nabla^2 u \|^2_{\mathbf{L}^2   }
       + 2\, \| \nabla w \|^2_{\mathbf{L}^\infty   } \| \nabla u \|^2_{\mathbf{L}^2   }
       + \| \nabla^2 w \|^2_{\mathbf{L}^2   } \| u \|^2_{\mathbf{L}^\infty   }
   \Big)
\nonumber
\\
 & \leq &
   c
   \Big( \| w \|^2_{\mathbf{H}^2   } \| \nabla^2 u \|^2_{\mathbf{L}^2   }
       + 2\, \| w \|^2_{\mathbf{H}^3   } \| \nabla u \|^2_{\mathbf{L}^2   }
       + \| \nabla^2 w \|^2_{\mathbf{L}^2   } \| u \|^2_{\mathbf{H}^2   }
   \Big)
\nonumber
\\
\end{eqnarray}
with a constant $c$ independent of $u$ and $w$.
On combining \eqref{eq.B.cont.1} and \eqref{eq.B.cont.3} we deduce that
\begin{eqnarray*}
\lefteqn{
   \| \mathbf{B} (w,u) \|^2_{L^2 (I,\mathbf{H}^1   )}
 \, \leq \,
   c \Big(
   \| w \|^2_{C (I,\mathbf{H}^2   )} \| \nabla^2 u \|^2_{L^2 (I,\mathbf{L}^2   )}
}
\\
 & + &
   2\, \| w \|^2_{L^2 (I,\mathbf{H}^3   )} \| \nabla u \|^2_{C (I,\mathbf{L}^2   )}
 + \| \nabla^2 w \|^2_{L^2 (I,\mathbf{L}^2   )} \| u \|^2_{C (I,\mathbf{H}^2   )}
   \Big),
\end{eqnarray*}
i.e.,
   $\mathbf{B} (w,\cdot)$ maps $B^{2,0,0}_\mathrm{vel} (I)$ continuously into $B^{0,0,0}_\mathrm{for} (I)$.

Since the space $B^{k+2,0,0}_\mathrm{vel} (I) $ is continuously embedded both into
   $C (I,\mathbf{H}^{k+2}   )$ and
   $L^2 (I,\mathbf{H}^{k+3}   )$,
it is continuously embedded into
   $C (I,\mathbf{W}^{k,\infty}   )$ and
   $L^2 (I,\mathbf{W}^{k+1,\infty}   )$,
which is due to the Sobolev embedding theorem.
If $|\alpha| = k$ then, similarly to
   \eqref{eq.Leib} and
   \eqref{eq.B.cont.3},
we get
\begin{equation}
\label{eq.Leib.k}
   \partial^\alpha_x \mathbf{B} (w,u)
 = \sum_{\beta+\gamma = \alpha}
   c_{\beta,\gamma}\, \mathbf{B} (\partial^\beta_x w, \partial^\gamma_x u),
\end{equation}
and
\begin{eqnarray*}
\lefteqn{
   \| \nabla^{k'} \mathbf{B} (w,u) \|^2_{\mathbf{L}^2   }
}
\\
 & \leq &
   \sum_{l=0}^{k'}
   c_{k',l}
   \Big(
   \| \nabla^l w \|^2_{\mathbf{L}^\infty   } \| \nabla^{k'+1-l} u \|^2_{\mathbf{L}^2   }
 + \| \nabla^{k'+1-l} w \|^2_{\mathbf{L}^2   } \| \nabla^l u \|^2_{\mathbf{L}^\infty   }
   \Big)
\\
 & \leq &
   \sum_{l=0}^{k'}
   c_{k',l}
   \Big(
   \| w \|^2_{\mathbf{H}^{l+2}   } \| \nabla^{k'+1-l} u \|^2_{\mathbf{L}^2   }
 + \| \nabla^{k'+1-l} w \|^2_{\mathbf{L}^2   } \| u \|^2_{\mathbf{H}^{l+2}   }
   \Big)
\end{eqnarray*}
for all $0 \leq k' \leq k+1$,
   the coefficients $c_{\beta,\gamma}$ and $c_{k',l}$ being of binomial type.
Hence it follows that
\begin{equation}
\label{eq.B.cont.6}
\begin{array}{rcl}
 &
 &
   \| \nabla^{k'} \mathbf{B} (w,u) \|^2_{C (I,\mathbf{L}^2   )}
\\
 & \leq
 &
   c\,
   \Big(
   \| w \|^2_{C (I,\mathbf{H}^{k'+2}   )} \| u \|^2_{C (I,\mathbf{H}^{k'+1}   )}
 + \| w \|^2_{C (I,\mathbf{H}^{k'+1}   )} \| u \|^2_{C (I,\mathbf{H}^{k'+2}   )}
   \Big)
\end{array}
\end{equation}
and
\begin{equation}
\label{eq.B.cont.7}
\begin{array}{rcl}
 &
 &
   \| \nabla^{k'+1} \mathbf{B} ( w,u) \|^2_{L^2 (I,\mathbf{L}^2   )}
\\
 & \leq
 &
   c\,
   \Big(
   \| w \|^2_{L^2 (I,\mathbf{H}^{k'+3}   )} \| u \|^2_{C (I,\mathbf{H}^{k'+2}   )}
 + \| w \|^2_{C (I,\mathbf{H}^{k'+2}   )} \| u \|^2_{L^2 (I,\mathbf{H}^{k'+3}   )}
   \Big)
\end{array}
\end{equation}
whenever $0 \leq k' \leq k$, i.e., $\mathbf{B} (w,\cdot)$ maps
   $B^{k+2,0,0}_\mathrm{vel} (I)$ continuously into
   $B^{k,0,0}_\mathrm{for} (I)$
for all $k \in {\mathbb Z}_+$.
The constant $c$ does not depend on $u$ and $v$ and it need not be the same in different applications.

If
   $k \in {\mathbb Z}_+$ and
   $s \geq 2$,
then the space $B^{k+2,2(s-1),s-1}_\mathrm{vel} (I)$ is continuously embedded into
   $C (I,\mathbf{H}^{k+2s}   )$ and
   $L^2 (I,\mathbf{H}^{k+2s+1}   )$.
By the Sobolev embedding theorem, it is continuously embedded into the spaces
   $C (I,\mathbf{W}^{k+2(s-1),\infty}   )$ and
   $L^2 (I,\mathbf{W}^{k+2(s-1)+1,\infty}   )$.
Moreover, each derivative $\partial^j_t$ belongs both to
   $B^{k,2(s-j),s-j}_\mathrm{vel} (I)$ and
   $C (I,\mathbf{H}^{k+2(s-j)}   )$,
   $L^2 (I,\mathbf{H}^{k+1+2(s-j)}   )$,
and so it also belongs to the spaces
   $C (I,\mathbf{W}^{k+2(s-1-j),\infty}   )$ and
   $L^2 (I,\mathbf{W}^{k+1+2(s-1-j),\infty}   )$.

If
   $|\alpha'| = k' \leq k+1$ and
   $|\alpha| + 2j \leq 2(s-1)$,
then,
   similarly to \eqref{eq.Leib.k},
we get with binomial type coefficients $c_{\beta,\gamma}$ and $C_j^l$
\begin{equation}
\label{eq.Leib.k.s}
   \partial^{\alpha'+\alpha}_x \partial_t^j \mathbf{B} (w,u)
 =
   \sum_{\beta+\gamma = \alpha'+\alpha}
   \sum_{l=0}^j
   c_{\beta,\gamma} C_j^l\,
   \mathbf{B} (\partial^\beta_x \partial_t^{l} w, \partial^\gamma_x \partial_t^{j-l} u),
\end{equation}
and
\begin{eqnarray*}
\lefteqn{
   \| \partial^{\alpha'+\alpha}_x \partial_t^j \mathbf{B} (w,u) \|^2_{\mathbf{L}^2   }
}
\\
 & \leq &
   c
   \sum_{\beta+\gamma = \alpha'+\alpha \atop 0 \leq l \leq j}
   c_{\beta,\gamma} C_j^l
   \Big(
   \| \partial_t^l w \|^2_{\mathbf{H}^{|\beta|+2}   }
   \| \partial_t^{j-l} u \|^2_{\mathbf{H}^{|\gamma|+1}   }
 + \| \partial_t^l w \|^2_{\mathbf{H}^{|\beta|+1}   }
   \| \partial_t^{j-l} u \|^2_{\mathbf{H}^{|\gamma|+2}   }
   \Big)
\end{eqnarray*}
for all multiindices $\alpha'$ of modulus $0 \leq k' \leq k+1$.
Therefore,
\begin{eqnarray}
\label{eq.B.cont.9}
\lefteqn{
   \| \partial^{\alpha'+\alpha}_x \partial_t^j \mathbf{B} (w,u) \|^2_{C (I,\mathbf{L}^2   )}
 \, \leq \,
   c
   \sum_{\beta+\gamma = \alpha'+\alpha \atop 0 \leq l \leq j}
   c_{\beta,\gamma} C_j^l
}
\nonumber
\\[-.2cm]
 \! & \! \times \! & \!
   \Big(
   \| \partial_t^l w \|^2_{C (I,\mathbf{H}^{|\beta|+2}   )}
   \| \partial_t^{j-l} u \|^2_{C (I,\mathbf{H}^{|\gamma|+1}   )}
 + \|\partial_t^l w \|^2_{C (I,\mathbf{H}^{|\beta|+1}   )}
   \|\partial_t^{j-l} u \|^2_{C (I,\mathbf{H}^{|\gamma|+2}   )}
   \Big),
\nonumber
\\
\end{eqnarray}
provided $0 \leq k' \leq k$, and
\begin{eqnarray}
\label{eq.B.cont.10}
\lefteqn{
   \| \partial^{\alpha'+\alpha}_x \partial_t^j \mathbf{B} (w,u) \|^2_{L^2 (I,\mathbf{L}^2   )}
 \, \leq \,
   c
   \sum_{\beta+\gamma = \alpha'+\alpha \atop 0 \leq l \leq j}
   c_{\beta,\gamma} C_j^l
}
\nonumber
\\[-.2cm]
 \! & \! \times \! & \!
   \Big(
   \| \partial_t^l w \|^2_{C (I,\mathbf{H}^{|\beta|+2}   )}
   \| \partial_t^{j-l} u \|^2_{L^2 (I,\mathbf{H}^{|\gamma|+1}   )}
 + \| \partial_t^l w \|^2_{L^2 (I,\mathbf{H}^{|\beta|+1}    )}
   \| \partial_t^{j-l} u \|^2_{C (I,\mathbf{H}^{|\gamma|+2}   )}
   \Big),
\nonumber
\\
\end{eqnarray}
provided $0 \leq k' \leq k+1$.
As
$$
\begin{array}{rcl}
   |\beta|
 & \leq
 & k-k' + 2(s-j-1),
\\
   |\gamma|
 & \leq
 & k-k' + 2(s-j-1),
\end{array}
$$
if
   $k' \leq k$ and
   $0 \leq j \leq s-1$
in \eqref{eq.B.cont.9}, and
$$
\begin{array}{rcl}
   |\beta|
 & \leq
 & k+1-k' + 2(s-j-1),
\\
   |\gamma|
 & \leq
 & k+1-k' + 2(s-j-1),
 \end{array}
 $$
 if
    $k' \leq k+1$ and
    $0 \leq j \leq s-1$
in \eqref{eq.B.cont.10}, we see that $\mathbf{B} (w,\cdot)$ maps the space
   $B^{k+2,2(s-1),s-1}_\mathrm{vel} (I)$ continuously into
   $B^{k,2(s-1),s-1}_\mathrm{for} (I)$
for all
   $k \in {\mathbb Z}_+$ and
   $s \in \mathbb N$.

The boundedness of the operators
$
  \mathbf{B} (w,\cdot) :
  B^{k,2s,s}_\mathrm{vel} (I)
 \to
   B^{k,2(s-1),s-1}_\mathrm{for} (I),
$
now follows from the continuous embedding
   $B^{k,2s,s}_\mathrm{vel} (I) \hookrightarrow  B^{k+2,2(s-1),s-1}_\mathrm{vel} (I)$
which is valid for $s \in \mathbb N$.

Finally, since
   the bilinear form $\mathbf{B}$ is symmetric and
   $\mathbf{B} (u,u) = 2 \mathbf{D} (u)$,
we easily obtain
\begin{equation}
\label{eq.M.diff}
   \mathbf{D} (u) - \mathbf{D} (u_0)
 = \mathbf{B} (u_0, u-u_0) + (1/2)\, \mathbf{B} (u-u_0, u-u_0).
\end{equation}
Therefore, by the continuity of the mapping $B (w,\cdot)$,
\begin{eqnarray}
\label{eq.M.diff.norm}
\lefteqn{
   \| \mathbf{D} (u) - \mathbf{D} (u_0) \|_{B^{k,2(s-1),s-1}_\mathrm{for} (I)}
 \, \leq \,
   c (k,s)
}
\nonumber
\\
 \!\! & \!\! \times \!\! & \!\!
   \Big( \| u_0 \|_{B^{k+2,2(s-1),s-1}_\mathrm{vel} (I)} \| u-u_0 \|_{B^{k+2,2(s-1),s-1}_\mathrm{vel} (I)}
       + \frac{1}{2} \| u-u_0 \|^2_{B^{k+2,2(s-1),s-1}_\mathrm{vel} (I)}
   \Big)
\nonumber
\\
\end{eqnarray}
with a positive constant $c (k,s)$ independent of $u$ and $u_0$, i.e., the nonlinear operator $\mathbf{D}$
maps
   $B^{k+2,2(s-1),s-1}_\mathrm{vel} (I)$ continuously into
   $B^{k,2(s-1),s-1}_\mathrm{for} (I)$.
\end{proof}

\begin{theorem}
\label{t.exist.NS.lin.strong}
Let
   $s \in \mathbb N$,
   $k \in {\mathbb Z}_+$
and
   $w \in B^{k,2s,s}_\mathrm{vel} (I)$.
Then \eqref{eq.NS.lin} induces a bijective continuous linear mapping
\begin{equation}
\label{eq.map.Aw}
   \mathcal{A}_w :
   B^{k,2s,s}_\mathrm{vel} (I) \times B^{k+1,2(s-1),s-1}_\mathrm{pre} (I)
\to
   B^{k,2(s-1),s-1}_\mathrm{for} (I) \times V_{2s+k},
\end{equation}
which admits a continuous inverse $\mathcal{A}^{-1}_w$.
\end{theorem}

\begin{proof}
Indeed, the continuity of $\mathcal{A}_w$ follows from Lemma \ref{l.NS.cont}.
Let
$$
\begin{array}{rcl}
   (u,p)
 & \in
 & B^{k,2s,s}_\mathrm{vel} (I) \times B^{k+1,2(s-1),s-1}_\mathrm{pre} (I),
\\
   \mathcal{A}_w (u,p) = (f,u_0)
 & \in
 & B^{k,2(s-1),s-1}_\mathrm{for} (I) \times V_{k+2s}.
\end{array}
$$
The integration by parts with the use of  \eqref{eq.deRham} yields
\begin{equation}
\label{eq.by.parts.1}
   - (\varDelta u,u)_{\mathbf{L}^2   }
 = \| \nabla u (\cdot,t) \|^2_{\mathbf{L}^2   }.
\end{equation}
As $\mathrm{div}\, u = 0$ in ${\mathbb R}^3 \times [0,T]$, we see that
\begin{equation}
\label{eq.Duu.zero.1}
   (\nabla p,u)_{\mathbf{L}^2   }
 = - (p,\mathrm{div}\, u)_{L^2   }
 = 0.
\end{equation}
Formulas
   \eqref{eq.dt},
   \eqref{eq.Bu.zero},
   \eqref{eq.by.parts.1} and
   \eqref{eq.Duu.zero.1}
readily imply that $u$ is a weak solution to \eqref{eq.NS.lin}, i.e., \eqref{eq.NS.lin.weak} is
fulfilled.

By the Sobolev embedding theorem, see \eqref{eq.Sob.index}, the space 
$B^{k,2s,s}_\mathrm{vel} (I)$ 
is continuously embedded into $L^2 (\mathbf{L}^\infty   )$.
By Lemma \ref{l.NS.lin.unique}, if the data $(f,u_0)$  are zero then
   $u = 0$ and
   $\nabla p (x,t) = 0$
for each $t \in [0,T]$.
Hence it follows that the pressure $p$ is equal identically to a function $c (t)$ on the interval
$[0,T]$.
But then Proposition \ref{c.Sob.d} implies that $p \equiv 0$, and so operator $\mathcal{A}_w$ is
injective.

Let
   $(f,u_0)$ be arbitrary data in $B^{k,2(s-1),s-1}_\mathrm{for} (I) \times V_{2s+k}$
and
   $\{ u_M \}$ be the sequence of the corresponding Faedo-Galerkin approximations constructed
   in the proof of Theorem \ref{t.exist.NS.lin.weak}.
The Fourier coefficients
   $f^{m,j} (t) = \langle f (\cdot,t), v_{m,j} \rangle$
belong to $C^{s-1} [0,T] \cap H^s [0,T]$, and the components
   $a^{m,j}_{m',j'} (t)$
belong to $C^{s} [0,T] \cap H^{s+1} [0,T]$, see \eqref{eq.C} and \eqref{eq.CC}.
Since $w \in B^{k,2s,s}_\mathrm{vel} (I)$, formula \eqref{eq.C} shows that the entries of the matrix
$$
   \exp \Big( \int_0^t A^{(M)} (s) ds \Big)
$$
are actually of the class $C^{s+1} [0,T] \cap H^{s+2} [0,T]$.
So, all components of the vector $c_M (t)$ belong to $C^{s} [0,T] \cap H^{s+1} [0,T]$.
In particular, we have
   $u_M \in B^{k',2s,s}_\mathrm{vel} (I)$ and
   $\partial_t u_M \in B^{k',2(s-1),s-1}_\mathrm{for} (I)$
for each $k' \in {\mathbb N}$.

From \eqref{eq.ODU.lin.0} it follows that
$$
\begin{array}{rcl}
 \!\! & \!\!
 \!\! & \!\!
   \displaystyle
   (\partial_t^{j+1} u_M\!, v_{m,j})_{\mathbf{L}^2   }
   \!
 + \mu \!
   \sum_{|\alpha|=1}
   (\partial^\alpha \partial^j_t u_M\!, \partial^\alpha v_{m,j})_{\mathbf{L}^2   }
 + (\partial_t^j \mathbf{B} (w,u_M\!), v_{m,j})_{\mathbf{L}^2   }
\\
 \!\! & \!\!
   =
 \!\! & \!\!
   (\partial_t^j f, v_{m,j})_{\mathbf{L}^2   }
\end{array}
$$
for all $m \in \{ 0 \} \cup \mathbb{N}_{2,3}$ satisfying $0 \leq m \leq M$, and $j = 1, \ldots, J_m$.
Moreover, $u_M (\cdot,0) = u_{0,M}$.
If we multiply the foregoing equation corresponding to indices $m$ and $j$ by
   $\left( m (2 \pi / \ell)^2  \right)^r c^{m,j}_M$
with $r \in {\mathbb Z}_+$, then, after
   summation with respect to $m$ and $j$ and
   integration by parts with respect to the space variables using \eqref{eq.dt},
we obtain
\begin{eqnarray}
\label{eq.um.bound.OM.r}
\lefteqn{
   \frac{d}{dt}\, \| \nabla^r u_M (\cdot,t) \|^2_{\mathbf{L}^2   }
 + 2 \mu\, \| \nabla^{r+1} u_M (\cdot,t) \|^2_{\mathbf{L}^2   }
}
\nonumber
\\
 & = &
   2\, ( (-\varDelta)^{(r-1)/2} (\mathbf{B} (w,u_M) - f) (\cdot,t),
         (-\varDelta)^{(r+1)/2} u_M (\cdot,t) )_{\mathbf{L}^2   }
\nonumber
\\
\end{eqnarray}
for all $t \in [0,T]$.
Furthermore, using \eqref{eq.B.cont.1}, we get
\begin{equation}
\label{eq.f.differ.est}
   2\, | ((-\varDelta)^{(r-1)/2} f, (-\varDelta)^{(r+1)/2} u_M)_{\mathbf{L}^2   } |
 \leq
   \frac{2}{\mu} \| \nabla^{r-1} f \|^2_{\mathbf{L}^2   }
 + \frac{\mu}{2} \| \nabla^{r+1}   u_M \|^2 _{\mathbf{L}^2    } ,
\end{equation}
and similarly
\begin{eqnarray}
\label{eq.D.differ.est}
\lefteqn{
   2\, | (\mathbf{B} (w, u_M), (-\varDelta)^{r} u_M)_{\mathbf{L}^2   } |
}
\nonumber
\\
 & = &
  2 | ((-\varDelta)^{(r-1)/2} \mathbf{B} (w, u_M),
       (-\varDelta)^{(r+1)/2} u_M)_{\mathbf{L}^2   } |
\nonumber
\\
 & \leq &
\frac{c}{\mu} 
   \| (-\varDelta)^{(r-1)/2} \mathbf{B} (w, u_M) \|^2_{\mathbf{L}^2   }
 + 
  \frac{\mu}{c} \| \nabla^{r+1} u_M \|^2_{\mathbf{L}^2   }
\nonumber
\\
\end{eqnarray}
with an arbitrary positive constant $c$.

\begin{lemma}
\label{l.OM.bound}
Suppose $k \in {\mathbb Z}_+$.
If $(f,u_0) \in B^{k,0,0}_\mathrm{for} (I) \times V_{k+2}$ then
$$
   \| \nabla^{k'} u_M \|^2_{C (I,\mathbf{L}^2   )}
 + \mu\, \| \nabla^{k'+1} u_M \|^2_{L^2 (I,\mathbf{L}^2   )}
 \leq
   c_{k'} (\mu,w,f,u_0)
$$
for any $0 \leq k' \leq k+2$, the constants
   $c_{k'} (\mu,w,f,u_0) > 0$
depending on
   $k'$ and
   $\mu$
and the norms
   $\| w |_{B^{k,2,1}_\mathrm{vel} (I)}$,
   $\| f \|_{B^{k,0,0}_\mathrm{for} (I)}$,
   $\| u_0 \|_{V_{k+2}}$
but not on $M$.
\end{lemma}

\begin{proof}
We argue by induction.

First, let $k=0$.
Then, similarly to \eqref{eq.B.cont.1}, using the Sobolev embedding theorem we get
\begin{eqnarray}
\label{eq.B.cont.1M}
   \!
   \int_0^t
   \!
   \| \mathbf{B} (w,u_M) \|^2_{\mathbf{L}^2   } ds
 \!\! & \!\! \leq \!\! & \!\!
   \!
   \int_0^t
   \!
   (\| w \|^2_{\mathbf{L}^\infty   } \| \nabla u_M \|^2_{\mathbf{L}^2   }
  + \| \nabla w \|^2_{\mathbf{L}^\infty   } \| u_M \|^2_{\mathbf{L}^2   }
    )
    ds
\nonumber
\\
 \!\! & \!\! \leq \!\! & \!\!
   \!
   \int_0^t
   \!
   \| w \|^2_{\mathbf{H}^2   } \| \nabla u_M \|^2_{\mathbf{L}^2   }
   ds
 + \| w \|^2_{L^2 (I,\mathbf{H}^3   )} \| u_M \|^2_{C (I,\mathbf{L}^2   )}.
\nonumber
\\
\end{eqnarray}

From \eqref{eq.um.bound}, Lemma \ref{p.En.Est.u} and the Sobolev embedding theorem it
follows that
$$
   \| w \|^2_{L^2 (I,\mathbf{H}^3   )} \| u_M \|^2_{C (I,\mathbf{L}^2   )}
 \leq
    c (w)\, \| (f,u_0) \|^2_{0,\mu,T}
$$
where $c (w)$ is a positive constant depending on
   $\| w \|^2_{L^2 (I,\mathbf{H}^3   )}$.
Next, according to Lemma \ref{l.basis.V} and Proposition \ref{p.Vs}, for each $k' \in 
{\mathbb Z}_+$,
the vector $\varDelta^{k'/2} u_{0,M}$ is the $\mathbf{L}^2   \,$-orthogonal projection
of $\varDelta^{k'/2} u_{0}$ onto to the linear span of the finite system
$$
   \{ v_{m,j} \}_{m \in \{ 0 \} \cup \mathbb{N}_{2,3} \atop 0 \leq m \leq M}.
$$
Hence, using the properties of projection and 
\eqref{eq.by.parts.0}, we see that
\begin{equation}
\label{eq.Bessel}
   \| \nabla^{k'} u_{0,M} \|^2_{\mathbf{L}^2   }
 \leq
   \| \nabla^{k'} u_{0} \|^2_{\mathbf{L}^2   }
\end{equation}
for all nonnegative integers $k'$ and $M$.
Now, combining
   \eqref{eq.um.bound.OM.r} for $r=0$
and
   \eqref{eq.f.differ.est},
   \eqref{eq.D.differ.est},
   \eqref{eq.B.cont.1M}
with integration  over the interval $[0,t]$, we arrive at the estimate
\begin{eqnarray}
\label{eq.um.bound.OM.1a}
\lefteqn{
   \| \nabla u_M (\cdot, t) \|^2_{\mathbf{L}^2   }
 + \mu
   \int_0^t \| \nabla^2 u_M (\cdot, s) \|^2_{\mathbf{L}^2   } ds
}
\nonumber
\\
 & \leq &
   \| \nabla u_{0,M} \|^2_{\mathbf{L}^2   }
 + \frac{2}{\mu} \| f \|^2_{L^2 (I,\mathbf{L}^2   )}
 + c_{0,0}
 + \frac{4}{\mu}
   \int_0^t
   \| w \|^2_{{\mathbf H}^2   } \| \nabla u_M \|^2_{\mathbf{L}^2   }
   ds
\nonumber
\\
\end{eqnarray}
for all $t \in [0,T]$.
Here the constant $c_{1,0}$ depends on
   $\| w \|_{B^{0,2,1}_\mathrm{vel} (I)}$ and
   $\|(f,u_0) \|_{0,\mu,T}$,
only.

At this point inequality \eqref{eq.um.bound.OM.1a} and Gronwall's Lemma \ref{l.Perov}
yield 
\begin{equation}
\label{eq.um.bound.OM.strong.1}
   \| \nabla u_M \|^2_{C (I,\mathbf{L}^2   )}
 + \mu\, \| \nabla^2 u_M \|^2_{L^2 (I,\mathbf{L}^2   )}
 \leq
   c_{1} (\mu,w,f,u_0)
\end{equation}
with a constant $c_{1} (\mu,w,f,u_0)$ depending on
   $\mu$
and
   $\| w \|_{B^{0,2,1}_\mathrm{vel} (I)}$,
   $\| f \|_{B^{0,0,0}_\mathrm{for} (I)}$ and
   $\| u_0 \|_{V_1}$,
only.

Next, on combining
   \eqref{eq.um.bound.OM.r} for $r = 1$
and
   \eqref{eq.B.cont.3},
   \eqref{eq.f.differ.est},
   \eqref{eq.D.differ.est}
with integration  over the interval $[0,t]$ we obtain
\begin{eqnarray}
\label{eq.um.bound.OM.1b}
\lefteqn{
   \| \nabla^2 u_M (\cdot, t) \|^2_{\mathbf{L}^2   }
 + \mu \int_0^t \| \nabla^3 u_M (\cdot, s) \|^2_{\mathbf{L}^2   } ds
}
\nonumber
\\
 & \leq &
   \| \nabla^2 u_{0,M} \|^2_{\mathbf{L}^2   }
 + \frac{2}{\mu} \| \nabla f \|^2_{L^2 (I,\mathbf{L}^2   )}
 + c_{0,0}
 + c_{1,0}
   \frac{2}{\mu}
   \int_0^t
   \| w \|^2_{{\mathbf H}^2   } \| \nabla^2 u_M \|^2_{\mathbf{L}^2   }
   ds
\nonumber
\\
 & + &
   c_{1,0}
   \frac{2}{\mu}
   \Big( \| w \|^2_{L^2 (I,\mathbf{H}^3   )} \| \nabla u_M \|^2_{C (I,\mathbf{L}^2   )}
       + \| \nabla^2 w \|^2_{C (I,\mathbf{L}^2   )} \| u_M \|^2_{L^2 (I,\mathbf{H}^2   )}
   \Big).
\nonumber
\\
\end{eqnarray}
From inequalities
   \eqref{eq.Bessel},
   \eqref{eq.um.bound.OM.strong.1},
   \eqref{eq.um.bound.OM.1b}
and Gronwall's Lemma \ref{l.Perov} 
it follows readily that
$$
   \| \nabla^2 u_M \|^2_{C (I,\mathbf{L}^2   )}
 + \mu\, \| \nabla^3 u_M \|^2_{L^2 (I,\mathbf{L}^2   )}
 \leq
    c_{2} (\mu,w,f,u_0),
$$
where $c_{2} (\mu,w,f,u_0)$ is a constant depending on
   $\mu$
and
   $\| w \|_{B^{0,2,1}_\mathrm{vel} (I)}$,
   $\| f \|_{B^{0,0,0}_\mathrm{vel} (I)}$ and
   $\| u_0 \|_{V_2}$,
only.

Assume that the sequence $\{ u_M \}$ is bounded in the space
   $B^{k,2,1}_\mathrm{vel} (I)$,
for given data $(f,u_0) \in B^{k,0,0}_\mathrm{for} (I) \times V_{k+2}$,
   with $k = k' \in \mathbb N$,
i.e.,
\begin{equation}
\label{eq.um.bound.OM.strong.k.prime}
   \| \nabla^{k''} u_M \|^2_{C (I,\mathbf{L}^2   )}
 + \mu\, \| \nabla^{k''+1} u_M \|^2_{L^2 (I,\mathbf{L}^2   )}
 \leq
   c_{k''} (\mu,w,f,u_0),
\end{equation}
if $0 \leq k'' \leq k'+2$, where the constants $c_{k''} (\mu,w,f,u_0)$ depend on
   $\mu$
and the norms
   $\| w \|_{B^{k,2,1}_\mathrm{vel} (I)}$,
   $\| f \|_{B^{k,0,0}_\mathrm{for} (I)}$,
   $\| u_0 \|_{V_{k+2}}$
but not on $M$.
Then, combining
   \eqref{eq.f.differ.est},
   \eqref{eq.D.differ.est}
with integration over the time interval $[0,t]$, we get
\begin{eqnarray}
\label{eq.um.bound.OM.1k}
\lefteqn{
   \| \nabla^{k'+3} u_M  (\cdot,t) \|^2_{\mathbf{L}^2   }
 + \mu \int_0^t \| \nabla^{k'+4} u_M  (\cdot,s) \|^2_{\mathbf{L}^2   } ds
}
\nonumber
\\
 & \leq &
   \| \nabla^{k'+3} u_{0,M} \|^2_{\mathbf{L}^2   }
 + \| \nabla^{k'+2} f \|^2_{L^2 (I,\mathbf{L}^2 )}
 + \frac{2}{\mu} \| \nabla^{k'+2} \mathbf{B} (w,u_M) \|^2_{L^2 (I,\mathbf{L}^2 )}.
\nonumber
\\
\end{eqnarray}
In this way we need to evaluate the last summand on the right-hand side of 
\eqref{eq.um.bound.OM.1k}.
Note that the estimate at \eqref{eq.Leib.k} does not fit our purpose because it uses the 
symmetry between $u$ and $w$ in the definition of $\mathbf{B} (w,u)$.
Here the symmetry is lost because of the lack of information about the desired solution $u$.
Still, according to \eqref{eq.Leib.k},
\begin{eqnarray*}
\lefteqn{
   \| \nabla^{k'+2} \mathbf{B} (w,u_M) \|^2_{L^2 (I,\mathbf{L}^2   )}
 \, = \,
   \| \nabla^{k'+2} (u_M \cdot \nabla w + w \cdot \nabla u_M) \|^2_{L^2 (I,\mathbf{L}^2   )}
}
\\
 & \leq &
   c_{k'+2,0}\,
   \frac{2}{\mu}
   \int_0^t
   \| \nabla^{k'+3} u_M (\cdot,s) \|^2_{\mathbf{L}^2   }
   \| w (\cdot,s) \|^2_{\mathbf{H}^2   }
   ds
\hspace{2cm}
\\
 & + &
   \frac{2}{\mu}
   \sum_{j=1}^{k'+2}
   c_{k'+2,j}\,
   \| \nabla^{k'+3-j} u_M \|^2_{C (I,\mathbf{L}^2   )}
   \| w \|^2_{L^2 (I,\mathbf{H}^{j+2}   )}
\\
 & + &
   \frac{2}{\mu}
   \sum_{j=0}^{k'+1}
   c_{k'+2,j}\,
   \| \nabla^{k'+3-j} w \|^2_{C (I,\mathbf{L}^2   )}
   \| u_M \|^2_{L^2 (I,\mathbf{H}^{j+2}   )}
\\
 & + &
   c_{k'+2,k'+2}\,
   \frac{2}{\mu}\,
   \| \nabla w \|^2_{C (I,\mathbf{L}^\infty   )}
   \| \nabla^{k'+2} u \|^2_{L^2 (I,\mathbf{L}^{2}   )},
\end{eqnarray*}
where the last summand is distinguished because of the lack of symmetry between $w$ and $u$.
All terms on the right-hand side of this inequality can be estimated due to the inductive
assumption of \eqref{eq.um.bound.OM.strong.k.prime} and the Sobolev embedding theorem.
In particular,
$$
   \| \nabla w \|^2_{C (I,\mathbf{L}^\infty   )}
   \| \nabla^{k'+2} u \|^2_{L^2 (I,\mathbf{L}^{2}   )}
 \leq
   \mathrm{const}\,
   \| w \|^2_{C (I,\mathbf{H}^3   )}
   \| \nabla^{k'+2} u \|^2_{L^2 (I,\mathbf{L}^{2}   )}
$$
with a suitable Sobolev constant.
From
   \eqref{eq.Bessel},
   \eqref{eq.um.bound.OM.strong.k.prime} and
   \eqref{eq.um.bound.OM.1k},
it follows that
\begin{eqnarray}
\label{eq.um.bound.OM.1k.A}
\lefteqn{
   \| \nabla^{k'+3} u_M  (\cdot, t) \|^2_{\mathbf{L}^2   }
 + \mu \int_0^t \| \nabla^{k'+4} u_M  (\cdot,s) \|^2_{\mathbf{L}^2   } ds
 \, \leq \,
   \| \nabla^{k'+3} u_{0} \|^2_{\mathbf{L}^2   }
}
\nonumber
\\
 & + &
   \| \nabla^{k'+2} f \|^2_{L^2 (I,\mathbf{L}^2   )}
 + c_{k'+2,0}\,
   \frac{2}{\mu}
   \int_0^t
   \| \nabla^{k'+3} u_M (\cdot,s) \|^2_{\mathbf{L}^2   }
   \| w (\cdot,s) \|^2_{\mathbf{H}^2   }
   ds
\nonumber
\\
 & + &
   R_{k'+3} (\mu,w,f,u_0),
\nonumber
\\
\end{eqnarray}
for all $t \in [0,T]$, the remainder $R_{k'+3} (w,f,u_0)$ depends on
   $\mu$
and
   $\| w \|_{B^{k'+1,2,1}_\mathrm{vel} (I)}$,
   $\| f \|_{B^{k'+1,0,0}_\mathrm{for} (I)}$,
   $\| u_0 \|_{V_{k'+2}}$,
only.

As before,
   \eqref{eq.Bessel},
   \eqref{eq.um.bound.OM.strong.k.prime},
   \eqref{eq.um.bound.OM.1k.A}
and Gronwall's Lemma \ref{l.Perov} yield
$$
   \| \nabla^{k'+3} u_M \|^2_{C (I,\mathbf{L}^2   )}
 + \mu\, \| \nabla^{k'+4} u_M \|^2_{L^2 (I,\mathbf{L}^2   )}
 \leq
   c_{k'+3} (\mu,w,f,u_0),
$$
the constant $c_{k'+3} (w,f,u_0)$ depends on
   $\mu$
and
   $\| w \|_{B^{k'+1,2,1}_\mathrm{vel} (I)}$,
   $\| f \|_{B^{k'+1,0,0}_\mathrm{for} (I)}$ and
   $\| u_0 \|_{V_{k'+3}}$
but not on the index $M$.
When combined with the induction hypothesis of \eqref{eq.um.bound.OM.strong.k.prime}, the
latter estimate implies that the assertion of the lemma is true for all $k \in {\mathbb Z}_+$.
\end{proof}

The bounds of Lemma \ref{l.OM.bound} mean precisely that the sequence $\{ u_M \}$ is bounded
in the space
   $C (I,\mathbf{H}^{k+2}   )\cap L^2 (I,\mathbf{H}^{k+3}   )$
if
   the data $(f,u_0)$ belong to $B^{k,0,0}_\mathrm{for} (I) \times V_{k+2}$.
Hence it follows that we may extract a subsequence $\{ u_{N'} \}$, such that

1) For any multiindex $\alpha$ satisfying $|\alpha| \leq k+3$, the sequence 
$\{ \partial^{\alpha}_x  u_{N'} \}$
converges weakly in $L^2 (I,\mathbf{L}^{2}   )$.

2) 
The sequence $\{   u_{N'} \} $ converges weakly-$^{\ast}$ in $L^\infty (I,\mathbf{H}^{k+2} )
 \cap L^2 (I,\mathbf{H}^{k+3}   )$ to an element $u$.

By the Sobolev embedding theorem, the space $B^{k,2,1}_\mathrm{vel} (I)$ is continuously 
embedded into
   $L^2 (I,\mathbf{L}^\infty   )$,
and so $w \in L^2 (I,\mathbf{L}^\infty   )$.
Therefore, by Theorem \ref{t.exist.NS.lin.weak}, the limit vector field $u$ is a unique 
solution to equations \eqref{eq.NS.lin.weak} of Bochner class
   $L^\infty (I,V_{k+2}) \cap L^2 (I,V_{k+3}) \cap C (I,V_0)$.
Moreover, by \eqref{eq.B.cont.6}, the field $\mathbf{B} (w, u)$ belongs to
   $L^\infty (I,\mathbf{H}^k   ) \cap L^2 (I,\mathbf{H}^{k+1}   )$.

Actually, Proposition \ref{p.Vs} and  \eqref{eq.NS.lin.weak} imply that
\begin{equation}
\label{eq.um.bound.OM.1t}
   \partial_t u
 = \mu \varDelta u + \mathbf{P} (f - \mathbf{B}(w, u))
\end{equation}
in the sense of distributions in ${\mathbb R}^3\times (0,T)$.
According to Lemma \ref{l.projection.commute}, the projection $\mathbf{P}$ maps
   $L^\infty (I,\mathbf{H}^k   ) \cap L^2 (I,\mathbf{H}^{k+1}   )$
continuously into
   $L^\infty (I,V_k) \cap L^2 (I,V_{k+1})$.
Hence, \eqref{eq.um.bound.OM.1t} implies that the derivative $\partial_t u$ belongs to
   $L^\infty (I,V_k) \cap L^2 (I,V_{k+1})$,
too.

Clearly,
   $\partial^\alpha_x \partial_t u \in L^2 (I,V_1')$ and
   $\partial^\alpha_x u \in L^2 (I,V_1)$
for all multiindices $\alpha \in {\mathbb Z}^3_+$ satisfying $|\alpha| \leq k+2$.
Hence,  Lemma \ref{l.Lions} yields
   $\partial^\alpha_x u \in C^2 (I,\mathbf{L}^{2}   )$
if $|\alpha| \leq k+2$.
Applying \eqref{eq.B.cont.6}, we see that
   $\mathbf{B} (w, u) \in C (I,\mathbf{H}^k   ) \cap L^2 (I,\mathbf{H}^{k+1}   )$.
According to Lemma \ref{l.projection.commute}, the projection $\mathbf{P}$ maps the space
   $C (I,\mathbf{H}^k   ) \cap L^2 (I,\mathbf{H}^{k+1}   )$
continuously into
   $C (I,V_k) \cap L^2 (I,V_{k+1})$.
So,
   $\partial_t u \in C (I,V_k) \cap L^2 (I,V_{k+1})$,
too, see \eqref{eq.um.bound.OM.1t}.

We have thus proved that \eqref{eq.NS.lin.weak} admits a unique solution
   $u \in  B^{k,2,1}_\mathrm{vel} (I)$
for any data $(f,u_0) \in B^{k,0,0}_\mathrm{for} (I) \times V_{k+2}$.
On the other hand, according to Proposition \ref{p.Vs} the vector field
   $(I - \mathbf{P}) (f - \mathbf{B} (w, u))$
belongs to $B^{k,0,0}_\mathrm{for} (I) \cap \ker (\mathrm{rot})$.
Hence, from Proposition \ref{c.Sob.d} we conclude that there is a unique function
   $p \in B^{k+1,0,0}_\mathrm{pre} (I)$
such that
\begin{equation}
\label{eq.nabla.p}
   \nabla p
 = (I - \mathbf{P}) (f - \mathbf{B}(w, u))
\end{equation}
in ${\mathbb R}^3 \times [0,T]$.
On adding equalities \eqref{eq.um.bound.OM.1t} and \eqref{eq.nabla.p} we readily deduce that
the pair
   $(u,p) \in B^{k,2,1}_\mathrm{vel} (I) \times B^{k+1,0,0}_\mathrm{pre} (I)$
is a unique solution to \eqref{eq.NS.lin} related to the data
   $(f,u_0) \in B^{k,0,0}_\mathrm{for} (I) \times V_{k+2}$,
i.e., the statement of the theorem concerning the surjectivity of the mapping $\mathcal{A}_w$
holds for $s=1$ and for any $k \in {\mathbb Z}_+$.

We finish the proof of the theorem with induction in $s \in \mathbb N$.
More precisely, assume that the assertion of the theorem concerning the surjectivity of the
mapping $\mathcal{A}_w$ holds for
   some $s = s' \in \mathbb N$ and
   any $k \in {\mathbb Z}_+$.
Let
   $(f,u_0) \in B^{k,2s',s'}_\mathrm{for} (I) \times V_{2 (s'+1)+k}$.
As
$$
   B^{k,2s',s'}_\mathrm{for} (I) \times V_{2 (s'+1)+k}
 \hookrightarrow
   B^{k+2,2(s'-1),s'-1}_\mathrm{for} (I) \times V_{2s'+k+2},
$$
we see that according to the induction assumption there is a unique solution $(u,p)$ to
\eqref{eq.NS.lin} which belongs to
   $ B^{k+2,2s',s'}_\mathrm{vel} (I) \times B^{k+3,2(s'-1),s'-1}_\mathrm{pre} (I)$.

By Lemmata \ref{l.projection.commute} and \ref{l.NS.cont}, the fields
   $\varDelta u$,
   $\mathbf{B}(w,u)$ and
   $\mathbf{P} (f - \mathbf{B}(w,u))$
belong to
   $B^{k,2s',s'}_\mathrm{for} (I)$,
and so the derivative $\partial_t u$ is in this space, too, because of \eqref{eq.um.bound.OM.1t}.
As a consequence, \eqref{eq.nabla.p} implies
   $\nabla p \in B^{k,2s',s'}_\mathrm{for} (I)$,
and so
   $p \in B^{k+1,2s',s'}_\mathrm{pre} (I)$.

Thus, the pair $(u,p)$ actually belongs to
   $B^{k,2(s'+1),s'+1}_\mathrm{vel} (I) \times B^{k+1,2s',s'}_\mathrm{pre} (I)$,
i.e., the mapping $\mathcal{A}_w$ of \eqref{eq.map.Aw} is surjective for all $k \in {\mathbb Z}_+$
and $s \in {\mathbb N}$.

Finally, as the mapping $\mathcal{A}_w$ is bijective and continuous, the continuity of the inverse
$\mathcal{A}^{-1}_w$ follows from the inverse mapping theorem for Banach spaces.
\end{proof}

Since problem \eqref{eq.NS.lin} is a linearisation of the Navier-Stokes equations at an arbitrary
vector field $w$, it follows from Theorem \ref{t.exist.NS.lin.strong} that the nonlinear mapping
given by the Navier-Stokes equations is locally invertible.
The implicit function theory for Banach spaces even implies that the local inverse mappings can
be obtained from the contraction principle of Banach.
In this way we obtain what we shall call the open mapping theorems for problem \eqref{eq.NS}.

\begin{theorem}
\label{t.open.NS.short}
Let
   $s \in \mathbb N$ and
   $k \in {\mathbb Z}_+$.
Then \eqref{eq.NS} induces an injective continuous nonlinear mapping
\begin{equation}
\label{eq.map.A}
   \mathcal{A} :
   B^{k,2s,s}_\mathrm{vel} (I) \times B^{k+1,2(s-1),s-1}_\mathrm{pre} (I)
\to
   B^{k,2(s-1),s-1}_\mathrm{for} (I) \times V_{2s+k}
\end{equation}
which is moreover open.
\end{theorem}

The principal significance of the theorem is in the assertion that for each point
   $(u_0,p_0) \in B^{k,2s,s}_\mathrm{vel} (I) \times B^{k+1,2(s-1),s-1}_\mathrm{pre} (I)$
there is a neighbourhood $\mathcal{V}$ of the image $\mathcal{A} (u_0,p_0)$ in
   $B^{k,2(s-1),s-1}_\mathrm{for} (I) \times V_{2s+k}$,
such that $\mathcal{A}$ is a homeomorphism of the open set
   $\mathcal{U} := \mathcal{A}^{-1} (\mathcal{V})$
onto $\mathcal{V}$.

\begin{proof}
Indeed, the continuity of the mapping $\mathcal{A}$ is clear from Lemma \ref{l.NS.cont}.
Moreover, suppose that
$$
\begin{array}{rcl}
   (u,p)
 & \in
 & B^{k,2s,s}_\mathrm{vel} (I) \times B^{k+1,2(s-1),s-1}_\mathrm{pre} (I),
\\
   \mathcal{A} (u,p)
 \, = \, (f,u_0)
 & \in
 & B^{k,2(s-1),s-1}_\mathrm{for} (I) \times V_{k+2s}.
\end{array}
$$
As in the proof of Theorem \ref{t.exist.NS.lin.strong}, formulas
   \eqref{eq.dt},
   \eqref{eq.Bu.zero},
   \eqref{eq.by.parts.1} and
   \eqref{eq.Duu.zero.1}
imply that \eqref{eq.NS.weak} is fulfilled,
i.e., $u$ is a weak solution to equations \eqref{eq.NS}.

By the Sobolev embedding theorem, see \eqref{eq.Sob.index}, the space
   $B^{k,2s,s}_\mathrm{vel} (I) $
is continuously embedded into $L^2 (I,\mathbf{L}^\infty ({\mathbb R}^3))$.
Hence, Theorem \ref{t.NS.unique} shows that if
   $(u',p')$ and
   $(u'',p'')$
belong to
   $B^{k,2s,s}_\mathrm{vel} (I) \times B^{k+1,2(s-1),s-1}_\mathrm{pre} (I)$
and
   $\mathcal{A} (u',p') = \mathcal{A} (u'',p'')$
then
   $u' = u''$ and
   $\nabla (p' - p'') (\cdot,t) = 0$
for all $t \in [0,T]$.
It follows that the difference $p'-p''$ is identically equal to a function $c (t)$ on
the segment $[0,T]$.
Since $p'-p'' \in C (I,\dot{H}^{k})$, we conclude by Proposition \ref{c.Sob.d} that
   $p'-p'' \equiv 0$.
So, the operator $\mathcal{A}$ of \eqref{eq.map.A} is injective.

Finally, equality \eqref{eq.M.diff} makes it evident that the Frech\'et derivative
   $\mathcal{A}'_{(w,p_0)}$
of the nonlinear mapping $\mathcal{A}$ at an arbitrary point
$$
   (w,p_0)
 \in B^{k,2s,s}_\mathrm{vel} (I) \times B^{k+1,2(s-1),s-1}_\mathrm{pre} (I)
$$
coincides with the continuous linear mapping $\mathcal{A}_{w}$ of \eqref{eq.map.Aw}.
By Theorem \ref{t.exist.NS.lin.strong}, $\mathcal{A}_{w}$ is an invertible continuous linear
mapping from
   $B^{k,2s,s}_\mathrm{vel} (I) \times B^{k+1,2(s-1),s-1}_\mathrm{pre} (I)$ to
   $B^{k,2(s-1),s-1}_\mathrm{for} (I) \times V_{k+2s}$.
Both the openness of the mapping $\mathcal{A}$ and the continuity of its local inverse mapping
now follow from the implicit function theorem for Banach spaces,
   see for instance  \cite[Theorem 5.2.3, p.~101]{Ham82}. 
	Actually, we need the following simple 
corollary of this theorem: Let $A: B_1 \to B_2$ be an  everywhere defined 
	smooth (admitting the Fr\'echet derivative at each point) map 
	between Banach spaces $B_1$, $B_2$. If for some point $v_0$ the Fr\'echet 
	derivative $A'_{|v_0}$ is a linear continuously invertible map then we can find 
	a neighbourhood $V$ of $v_0$ and 	a neighbourhood $U$ of $g_0 = A(v_0)$ 
	such that the map $A$ gives a one-to-one map of $V$ onto $U$ and the (local) 
	inverse map $A^{-1}: U \to V$ is continuous and smooth. 
\end{proof}

Thus, everywhere below we identify the Navier-Stokes equations 
with the following operator equation related to the mapping \eqref{eq.map.A}:
given data $(f,u_0) \in B^{k,2(s-1),s-1}_\mathrm{for} (I) \times V_{k+2s}$ find 
a pair $(u,p) \in B^{k,2s,s}_\mathrm{vel} (I) \times B^{k+1,2(s-1),s-1}_\mathrm{pre} (I)$ 
satisfying 
\begin{equation} \label{eq.NS.map}
{\mathcal A} (u,p) = (f,u_0).
\end{equation}
In this way the injectivity of the mapping ${\mathcal A}$ corresponds to the Uniqueness 
Theorem for the Navier-Stokes equations and the surjectivity of the mapping 
${\mathcal A}$ corresponds to the Existence Theorem. 
In particular, solutions to \eqref{eq.NS.map} are the so-called {\it strong} solutions 
to \eqref{eq.NS.weak},  see \cite[Ch. 3, \S 3.6]{Tema79}. 

Next, we set
$
   V_\infty = \mathbf{C}^\infty    \cap \ker (\mathrm{div}).
$
Moreover, given a Fr\'{e}chet space $\mathcal{F}$, we denote by $C^\infty (I,{\mathcal F})$ the
space of all infinitely differentiable functions of $t \in I$ with values in $\mathcal{F}$.

\begin{corollary}
\label{c.open.NS.short}
Equations \eqref{eq.NS} induce an injective continuous nonlinear mapping
\begin{equation}
\label{eq.map.A.Frechet}
   \mathcal{A} :
   C^\infty (I,V_\infty) \times C^\infty (I,\dot{C}^{\infty}   )
\to
   C^\infty (I,\mathbf{C}^{\infty}   ) \times V_\infty
\end{equation}
which is moreover open.
\end{corollary}

\begin{proof} It follows immediately from
   the Sobolev embedding theorem   that 
	$$
\begin{array}{rcl}
   V_\infty
 & =
 & \displaystyle
   \cap_{s=0}^\infty V_s,
\\
   C^\infty (I, V_\infty)
 & =
 & \displaystyle
   \cap_{s=1}^\infty B^{0,2s,s}_\mathrm{vel} (I),
\end{array}
   \quad
\begin{array}{rcl}
   C^\infty (I,\dot{C}^{\infty}    )
 & =
 & \displaystyle
   \cap_{s=0}^\infty B^{0,2s,s}_\mathrm{pre} (I),
\\
   C^\infty (I,\mathbf{C}^{\infty}   )
 & =
 & \displaystyle
   \cap_{s=0}^\infty B^{0,2s,s}_\mathrm{for} (I),
\end{array}
$$
see Proposition \ref{p.Vs}. Now, as the corresponding norms define {\it tame} 
systems of seminorms, the statement follows from Nash-Moser Theorem, see 
\cite[Part III, \S III.1, Theorem 1.1.1]{Ham82}. 
\end{proof}

Theorem \ref{t.open.NS.short} and Corollary \ref{c.open.NS.short} suggest a clear direction 
for the development of the topic, in which one takes into account the following property of 
the so-called clopen (closed and open) sets.

\begin{corollary}
\label{c.clopen}
The range of the mapping \eqref{eq.map.A} is closed if and only if it coincides with the whole
destination space. The statement is true for mapping \eqref{eq.map.A.Frechet}, too.
\end{corollary}

\begin{proof}
Since the destination space is convex, it is connected. As is known, the only clopen sets in 
a connected topological vector space are the empty set and the space itself. Hence, the range 
of the mapping $\mathcal{A}$ is closed if and only if it coincides with the whole
destination space.
\end{proof}

\section{A surjectivity criterion}
\label{s.surjective}

Inspired by
   \cite{Pro59},
   \cite{Serr62},
   \cite{Lady70} and
   \cite{Lion61,Lion69},
let us obtain a surjectivity criterion  for mapping 
\eqref{eq.map.A} and \eqref{eq.map.A.Frechet} in terms of 
$L^{\mathfrak{s}} (I,\mathbf{L}^{\mathfrak{r}}   )\,$-estimates 
for solutions to \eqref{eq.NS.map} via the data. 
The following theorem hints us that the surjectivity of mapping 
\eqref{eq.map.A}, induced by  \eqref{eq.NS}, is equivalent 
to a property similar to the {\it properness}, see for instance \cite{Sm65}.

\begin{theorem}
\label{t.LsLr}
Let
   $s \in \mathbb{N}$,
   $k \in {\mathbb Z}_+$
and the numbers $\mathfrak{r}$, $\mathfrak{s}$ satisfy \eqref{eq.s.r}. 
Then mapping \eqref{eq.map.A} is surjective 
if and only if, given subset   $S = S_\mathrm{vel} \times S_\mathrm{pre}$ of the product
   $B^{k,2s,s}_\mathrm{vel} (I) \times B^{k+1,2(s-1),s-1}_\mathrm{pre} (I)$ 
such that the image $\mathcal{A} (S)$ is precompact in the space
   $B^{k,2(s-1),s-1}_\mathrm{for} (I) \times V_{2s+k}$,
the set $S_\mathrm{vel}$ is bounded in the space
   $L^{\mathfrak{s}} (I,\mathbf{L}^{\mathfrak{r}}   )$.
\end{theorem}

\begin{proof} Let mapping \eqref{eq.map.A} be surjective. Then the range of this mapping 
is closed according to Theorem \ref{t.open.NS.short}. Fix a 
subset   $S = S_\mathrm{vel} \times S_\mathrm{pre}$ of the product
   $B^{k,2s,s}_\mathrm{vel} (I) \times B^{k+1,2(s-1),s-1}_\mathrm{pre} (I)$ 
such that the image $\mathcal{A} (S)$ is precompact in the space
   $B^{k,2(s-1),s-1}_\mathrm{for} (I) \times V_{2s+k}$. 
If the set $S_\mathrm{vel}$ is unbounded in the space
   $L^{\mathfrak{s}} (I,\mathbf{L}^{\mathfrak{r}}   )$ then 
	there is a sequence $\{ (u_k,p_k) \} \subset S$ such that
\begin{equation}
\label{eq.unbounded}
   \lim_{k \to \infty} \| u_k \|_{L^{\mathfrak{s}} 
(I,\mathbf{L}^{\mathfrak{r}}   )}
 = \infty.
\end{equation}
As the set $\mathcal{A} (S)$ is precompact in
   $B^{k,2(s-1),s-1}_\mathrm{for} (I) \times V_{2s+k}$,
we conclude that the corresponding sequence of data
   $\{ \mathcal{A} (u_k,p_k) = (f_k,u_{k,0})\}$
contains a subsequence $\{ (f_{k_m},u_{k_m,0})\}$ which converges to an element
   $(f,u_0) $ in this space. But the range of the map is closed and hence 
for the data $(f,u_0)$ there is a unique solution $(u,p)$ to 
\eqref{eq.NS.map} 
in the space
   $B^{k,2s,s}_\mathrm{vel} (I) \times B^{k+1,2(s-1),s-1}_\mathrm{pre} (I)$
and the sequence $\{ (u_{k_m},p_{k_m}) \}$ converges to $(u,p) $ in this space.
Therefore, $\{ (u_{k_m},p_{k_m}) \}$ is bounded in
   $B^{k,2s,s}_\mathrm{vel} (I) \times B^{k+1,2(s-1),s-1}_\mathrm{pre} (I)$
and this contradicts \eqref{eq.unbounded} because the space $B^{k,2s,s}_\mathrm{vel} (I)$ is
embedded continuously into the space $L^{\mathfrak{s}} 
(I,\mathbf{L}^{\mathfrak{r}}   )$ for any pair $\mathfrak{r}$, $\mathfrak{s}$ satisfying 
\eqref{eq.s.r}.

We continue with typical estimates for 
solutions to the Navier-Stokes equations \eqref{eq.NS.map}.
We emphasize again that the elements of the spaces in 
the consideration are already sufficiently regular. So, we need the estimates for proving 
the surjectivety of mapping \eqref{eq.map.A} but not for improving 
regularity of weak (Leray-Hopf) solutions.  

\begin{lemma}
\label{p.En.Est.u.strong}
If
   $(u,p) \in B^{0,2,1}_{\mathrm{vel}} (I) \times B^{1,0,0}_{\mathrm{pre}} (I)$
is a solution to the Navier-Stokes equations \eqref{eq.NS.map} with data
   $(f,u_0) \in B^{0,0,0}_{\mathrm{for}} (I) \times V_2$,
then
\begin{equation}
\label{eq.En.Est2imp}
   \| u \|_{0,\mu,T}
 \leq
   \| (f,u_0) \|_{0,\mu,T}
\end{equation}
and estimate \eqref{eq.En.Est2add} holds true for $u$.
\end{lemma}

\begin{proof}
For a solution  $(u,p)$ to \eqref{eq.NS.map} related to 
data $(f,u_0)$ within the declared function classes, the component   
$u$ belongs to $C (I, H^{2}) \cap L^2 (I, H^{3})$, and both $\partial_t u $ and $f$  
belong to $C (I, L^{2}) \cap L^2 (I, H^{1})$. Then  
we may calculate the inner product $({\mathcal A}u, u)_{\mathbf{L}^2}$
with the use of of \eqref{eq.Bu.zero} and Lemma \ref{l.Lions}, obtaining 
\begin{equation*}
 \frac{d}{dt} \|u\|^2_{\mathbf{L}^2}
 + \mu \| \nabla u\|^2_{\mathbf{L}^2}
  = \langle f , u \rangle,
\end{equation*}
Finally, applying Lemma \ref{p.En.Est.u}  with $w = 0$, we conclude 
the estimates of \eqref{eq.En.Est2imp} and \eqref{eq.En.Est2add} follow.
\end{proof}

Next, we obtain some estimates for the derivatives vector fields with respect to the space variables.

\begin{lemma}
\label{l.En.Est.Du.L2}
Let
   $k \in \mathbb{Z}_+$
and
   $\mathfrak{s}$,
   $\mathfrak{r}$
satisfy \eqref{eq.s.r}.
Then for any
   $\varepsilon > 0$
and for all
   $u \in \mathbf{H}^{2+k} $
it follows that
\begin{eqnarray}
\label{eq.En.Est3bbb}
\lefteqn{
   \| (- \varDelta)^{\frac{k}{2}} \mathbf{D} u \|^2_{\mathbf{L}^2   }
 \, \leq \,
   \varepsilon\, \| \nabla^{k+2} u \|^2_{\mathbf{L}^2    }
}
\nonumber
\\
 & + &
   c (k,\mathfrak{s},\mathfrak{r},\varepsilon)\, \| u \|^{\mathfrak{s}}_{\mathbf{L}^\mathfrak{r}   }
                                                 \| \nabla^{k+1} u \|^2_{\mathbf{L}^2   }
 + c (k,\mathfrak{s},\mathfrak{r})\, \| u \|^{2}_{\mathbf{L}^2   }
                                     \| u \|^{2}_{\mathbf{L}^\mathfrak{r}   }
 + c (k,\mathfrak{s},\mathfrak{r})\, \| u \|^{2}_{\mathbf{L}^2   }
\nonumber
\\
\end{eqnarray}
with positive constants depending on the parameters in parentheses and not necessarily the same in
diverse applications, the constants being independent of $u$.
\end{lemma}

\begin{proof}
On using
   the Leibniz rule,
   the H\"older inequality and
   Remark \ref{r.Delta.r}
we deduce that
\begin{equation}
\label{eq.En.Est3bu}
   \| (-\varDelta)^{\frac{k}{2}} \mathbf{D}  u \|^2_{\mathbf{L}^\mathfrak{r}   }
 \leq
   \sum_{j=0}^{k}
   C^k_j\,
   \| \nabla^{k+1-j} u \|^2_{\mathbf{L}^{\frac{2q}{q-1}}   }
   \| \nabla^{j} u \|^2_{\mathbf{L}^{2 q}    }
\end{equation}
with binomial type coefficients $C_k^j$ and any $q \in (1,\infty)$.

For $k = 0$ there are no other summands than that with $j=0$.
But for $k \geq 1$ we have to consider the items corresponding to $1 \leq j \leq k$, too.
The standard interpolation inequalities on compact manifolds
   (see for instance \cite[Theorem 2.2.1]{Ham82})
hint us that those summands which correspond to $1 \leq j \leq k$ could actually be estimated by
the item with $j=0$.
We realize this as follows:
For any $j$ satisfying $1 \leq j \leq k$ there are numbers
   $q > 1$ and
   $c > 0$
depending on $k$ and $j$ but not on $u$, such that
\begin{equation}
\label{eq.En.Est3y}
   \| \nabla^{k+1-j} u \|_{\mathbf{L}^{\frac{2 q}{q-1}}   }
   \| \nabla^{j} u \|_{\mathbf{L}^{2 q}   }
 \leq
   c \Big( \| \nabla^{k+1} u \|_{\mathbf{L}^{\frac{2 \mathfrak{r}}{\mathfrak{r}-2}}   }
              \| u \|_{\mathbf{L}^{\mathfrak{r}}   }
         + \| u \|_{\mathbf{L}^{2}   }
     \Big).
\end{equation}
Indeed, we may apply Gagliardo-Nirenberg inequality \eqref{eq.L-G-N} if we prove that for each
$1 \leq j \leq k$ there is a $q > 1$ depending on $k$ and $j$, such that the system of algebraic
equations
\begin{equation*}
\left\{
\begin{array}{rcl}
   \displaystyle
   \frac{1}{2q}
 & =
 & \displaystyle
   \frac{j}{3}
 + \Big( \frac{\mathfrak{r}-2}{2 \mathfrak{r}}-\frac{k+1}{3} \Big)\vartheta_1
 + \frac{1-\vartheta_1}{\mathfrak{r}},
\\
   \displaystyle
   \frac{q-1}{2q}
 & =
 & \displaystyle
   \frac{k+1-j}{3}
 + \Big( \frac{\mathfrak{r}-2}{2 \mathfrak{r}} -\frac{k+1}{3} \Big) \vartheta_2
 + \frac{1-\vartheta_2}{\mathfrak{r}}
\end{array}
\right.
\end{equation*}
admits solutions
$$
\begin{array}{rcl}
   \vartheta_1
 & \in
 & \displaystyle
   [\frac{j}{k+1},1),
\\
   \vartheta_2
 & \in
 & \displaystyle
   [\frac{k+1-j}{k+1},1).
\end{array}
$$
On adding these equations we see that
$$
   \frac{1}{2} - \frac{k+1}{3} - \frac{2}{\mathfrak{r}}
 = \Big( \frac{1}{2} - \frac{k+1}{3} - \frac{2}{\mathfrak{r}} \Big) (\vartheta_1 + \vartheta_2),
$$
i.e., the system is reduced to
$$
\left\{
\begin{array}{rcl}
   \vartheta_1 q (2 (k+1) \mathfrak{r }+ 12 - 3 \mathfrak{r})
 & =
 & 2 j\, \mathfrak{r} q + 6 q - 3 \mathfrak{r},
\\
   \vartheta_1 + \vartheta_2
 & =
 & 1.
\end{array}
\right.
$$
Choose
$
   \displaystyle
   \vartheta_1 = \frac{j}{k+1}
$
and
$
   \displaystyle
   \vartheta_2 = \frac{k+1-j}{k+1}
$
to obtain
$$
   q
 = q (k,j)
 = \frac{(k+1) \mathfrak{r}}{2 (k+1) + j (\mathfrak{r} -4)}.
$$
Since
   $\mathfrak{r} > n \geq 2$ and
   $1 \leq j \leq k$,
an easy calculation shows that
$$
\begin{array}{rcccccc}
   2 (k+1) + j (\mathfrak{r} - 4)
 & >
 & 2 (k+1) - 2j
 & \geq
 & 2
 & >
 & 0,
\\\
   (k+1) \mathfrak{r} - (2 (k+1) + j (\mathfrak{r} -4))
 & =
 & (k+1) (\mathfrak{r}-2) - j (\mathfrak{r} - 4)
 & >
 & 0,
 &
 &
\end{array}
$$
i.e., $q (k,j)> 1$ in this case, and so \eqref{eq.En.Est3y} holds true.

Therefore, if we choose $q (k,0) = \mathfrak{r} / 2 > 1$, the estimates of
   \eqref{eq.En.Est3bu} and
   \eqref{eq.En.Est3y}
readily yield
\begin{equation}
\label{eq.En.Est4b}
   \| (- \varDelta)^{\frac{k}{2}} \mathbf{D}  u \|^2_{\mathbf{L}^2   }
 \leq
   c (k,\mathfrak{r})
   \Big( \| \nabla^{k+1} u \|^2_{\mathbf{L}^{\frac{2 \mathfrak{r}}{\mathfrak{r}-2}}   }
         \| u \|^2_{\mathbf{L}^{\mathfrak{r}}   }
       + \| u \|^2_{\mathbf{L}^{2}   }
   \Big)
\end{equation}
with a constant $c (k,\mathfrak{r})$ independent on $u$.

Now, if
   $\mathfrak{s} = 2$ and
   $\mathfrak{r} = +\infty$,
then, obviously, we get
\begin{equation}
\label{eq.En.Est3uuu}
   c (k,\mathfrak{r})\,
   \| \nabla^{k+1} u \|^2_{\mathbf{L}^{\frac{2 \mathfrak{r}}{\mathfrak{r}-2}}   }
   \| u \|^2_{\mathbf{L}^{\mathfrak{r}}   }
 = c (k,\mathfrak{r})\,
   \| \nabla^{k+1} u \|^2_{\mathbf{L}^{2}   }
   \| u \|^2_{\mathbf{L}^{\infty}   }.
\end{equation}
If
   $\mathfrak{s} > 2$ and
   $3 < \mathfrak{r} < \infty$,
then we may again apply Gagliardo-Nirenberg inequality \eqref{eq.L-G-N} with
   $j_0 = 0$,
   $k_0 = 1$,
   $q_0 = r_0 = 2$,
   $0 < a = 3 / \mathfrak{r} < 1$
and
   $p_0 = 2 \mathfrak{r}/(\mathfrak{r}-2)$
to achieve
\begin{equation}
\label{eq.L-G-N.gamma}
   \| \nabla^{k+1} u \|_{\mathbf{L}^{\frac{2 \mathfrak{r}}{\mathfrak{r}-2}}   }
   \| u \|_{\mathbf{L}^{\mathfrak{r}}   }
 \leq
   c (\mathfrak{r})
   \Big( \| \nabla^{k+2} u \|^{\frac{3}{\mathfrak{r}}}_{\mathbf{L}^2   }\,
   \| \nabla^{k+1} u \|^{\frac{\mathfrak{r}-3}{\mathfrak{r}}}_{\mathbf{L}^2   }
 + \| u \|_{\mathbf{L}^{2}   }
   \Big)
   \| u \|_{\mathbf{L}^{\mathfrak{r}}   }
\end{equation}
with an appropriate Gagliardo-Nirenberg constant $c (\mathfrak{r})$ independent of $u$.

Since
$
   \displaystyle
   \mathfrak{s} = \frac{2 \mathfrak{r}}{\mathfrak{r}-3},
$
it follows from \eqref{eq.L-G-N.gamma} that
\begin{eqnarray}
\label{eq.En.Est3bbbb}
\lefteqn {
   c (k,\mathfrak{r})
   \| \nabla^{k+1} u \|^2_{\mathbf{L}^{\frac{2 \mathfrak{r}}{\mathfrak{r}-2}}   }
   \| u \|^2_{\mathbf{L}^{\mathfrak{r}}   }
}
\nonumber
\\
 & \leq &
   2 c (k,\mathfrak{r})
   \Big( \| \nabla^{k+2} u \|^{\frac{6}{\mathfrak{r}}}_{\mathbf{L}^2   }\,
         \| \nabla^{k+1} u \|^{\frac{2 (\mathfrak{r}-3)}{\mathfrak{r}}}_{\mathbf{L}^2   }\,
         \| u \|^2_{\mathbf{L}^{\mathfrak{r}}   }
       + \| u \|^2_{\mathbf{L}^{2}   }
         \| u \|^2_{\mathbf{L}^{\mathfrak{r}}    }
   \Big)
\nonumber
\\
 & \leq &
   \varepsilon\, \| \nabla^{k+2} u \|^2_{\mathbf{L}^2   }
 + \frac{c (k,\mathfrak{r})}{\varepsilon} \| \nabla^{k+1} u \|^{2}_{\mathbf{L}^2   }\,
                                          \| u \|^\mathfrak{s}_{\mathbf{L}^{\mathfrak{r}}    }
 + 2 c (k,\mathfrak{r})\, \| u \|^2_{\mathbf{L}^{2}   }
                          \| u \|^{2}_{\mathbf{L}^{\mathfrak{r}}    }
\nonumber
\\
\end{eqnarray}
with some positive constants independent of $u$ because of Young's inequality \eqref{eq.Young} applied
with
   $p_1 = \mathfrak{r} / 3$  and
   $p_2 = \mathfrak{r} / (\mathfrak{r}-3)$.

Now, inequalities
   \eqref{eq.En.Est4b},
   \eqref{eq.En.Est3uuu} and
   \eqref{eq.En.Est3bbbb}
imply \eqref{eq.En.Est3bbb} for all
   $3 < \mathfrak{r} \leq \infty$ and
   $2 \leq \mathfrak{s} = 2 \mathfrak{r} / (\mathfrak{r}-3) < \infty$,
as desired.
\end{proof}

We now introduce
\begin{equation*}
   \| (f,u_0) \|_{k,\mu,T}
 = \Big( \| \nabla^k u_0 \|^2_{\mathbf{L}^2   }
       + 4 \mu ^{-1} \| \nabla^{k-1} f \|^2_{L^2 (I, \mathbf{L}^2   )} \Big)^{1/2}
\end{equation*}
for $k \geq 1$.

\begin{lemma}
\label{t.En.Est.g.2}
Let
   $k \in \mathbb{Z}_+$ and the pair 
   $\mathfrak{s}$, $\mathfrak{r}$ satisfy \eqref{eq.s.r}.
If
   $(u,p) \in B^{k,2,1}_\mathrm{vel} (I) \times B^{k+1,0,0}_\mathrm{pre} (I)$
is a solution to the Navier-Stokes equations \eqref{eq.NS.map}  
corresponding to data
   $(f,u_0)$ in $B^{k,0,0}_\mathrm{for} (I) \times V_{k+2}$
then
\begin{eqnarray}
\label{eq.En.Est4}
   \| u \|_{j+1,\mu,T}
 & \leq &
   c_j ( (f,u_0), u),
\nonumber
\\
   \| \nabla^{j} \mathbf{D} u \|_{L^2 (I, \mathbf{L}^2   )}
 & \leq &
   c_j ( (f,u_0), u),
\nonumber
\\
   \| \nabla^{j} \partial_t u \|^2_{L^2 (I, \mathbf{L}^2   )}
 + \| \nabla^{j+1} p \|^2_{L^2 (I, \mathbf{L}^2   )}
 & \leq &
   c_j ( (f,u_0), u),
\nonumber
\\
\end{eqnarray}
for all $0 \leq j \leq k+1$, where the constants on the right-hand side depend on the norms 
   $\| (f,u_0) \|_{0,\mu,T}$,
   $\| (f,u_0) \|_{j+1,\mu,T}$ and
   $\| u \|_{L^\mathfrak{s} (I, \mathbf{L}^\mathfrak{r}   )}$
and need not be the same in diverse applications.
\end{lemma}

It is worth pointing out that the constants on the right-hand side of \eqref{eq.En.Est4} may  
also depend on $\mathfrak{s}$, $\mathfrak{r}$, $T$, $\mu$, etc., but we do not display this 
dependence in notation.

\begin{proof}
We first recall that
   $u \in C (I, \mathbf{H}^{k+2}   ) \cap L^2 (I, \mathbf{H}^{k+3}   )$,
   $u_0 \in \mathbf{H}^{k+2}   $
and
   $\nabla p, f \in C (I, \mathbf{H}^k   ) \cap L^2 (I, \mathbf{H}^{k+1}   )$
under the hypotheses of the lemma. 
Next, we see that in the sense of distributions we have
\begin{equation}
\label{eq.du.solution}
\left\{
\begin{array}{rclcl}
   (- \varDelta)^{\frac{j}{2}}
   (\partial_t - \mu \varDelta) u + \mathbf{D} u + \nabla p)
 & =
 & (- \varDelta)^{\frac{j}{2}} f
 & \mbox{in}
 & \mathbb{R}^3 \times (0,T),
\\
   (- \varDelta)^{\frac{j}{2}} u (x,0)
 & =
 & (- \varDelta)^{\frac{j}{2}} u_0 (x)
 & \mbox{for}
 & x \in \mathbb{R}^3
\end{array}
\right.
\end{equation}
 for all $0 \leq j \leq k+1$, if $(u,p)$ 
is a solution to \eqref{eq.NS.map}.

Integration by parts and Remark \ref{r.Delta.r} yield
\begin{equation}
\label{eq.by.parts.3}
   ( (- \varDelta)^{\frac{j}{2}} u, (- \varDelta)^{\frac{j+2}{2}} u)_{\mathbf{L}^2   }
 = \| (- \varDelta)^{\frac{j+1}{2}} u \|^2_{\mathbf{L}^2   }
 = \| \nabla^{j+1} u \|^2_{\mathbf{L}^2   }
\end{equation}
and similarly
\begin{equation}
\label{eq.dt.k}
   2\, (\partial_t (- \varDelta)^{\frac{j}{2}} u, (- \varDelta)^{\frac{j+2}{2}} u)_{\mathbf{L}^2   }
 = \frac{d}{dt}\, \| \nabla^{j+1} u \|^2_{\mathbf{L}^2   },
\end{equation}
cf. \eqref{eq.dt}.
Furthermore, as
   $\mathrm{rot}\, \nabla u = 0$ and
   $\mathrm{div}\, u =0$
in ${\mathbb R}^3 \times [0,T]$, we conclude that
\begin{eqnarray}
\label{eq.by.parts.3p}
\lefteqn{
   ((- \varDelta)^{\frac{j}{2}} \nabla p (\cdot,t), (- \varDelta)^{\frac{j+2}{2}} u (\cdot,t))_{\mathbf{L}^2   }
}
\nonumber
\\
 & = &
   \lim_{i \to \infty}
   ((- \varDelta)^{\frac{j}{2}} \nabla p_i (\cdot,t),
    (\mathrm{rot})^\ast \mathrm{rot}\, (- \varDelta)^{\frac{j}{2}} u (\cdot,t))_{\mathbf{L}^2   }
\nonumber
\\
 & = &
   \lim_{i \to \infty}
   ((- \varDelta)^{\frac{j}{2}} \mathrm{rot}\, \nabla p_i (\cdot,t),
    \mathrm{rot}\, (- \varDelta)^{\frac{j}{2}} u (\cdot,t))_{\mathbf{L}^2   }
\nonumber
\\
 & = &
   0
\nonumber
\\
\end{eqnarray}
for all $t \in [0,T]$, where
   $p_i (\cdot,t) \in H^{j+2} $
is any sequence approximating $p (\cdot, t)$ in $H^{j+1} $.

On combining
   \eqref{eq.du.solution},
   \eqref{eq.by.parts.3},
   \eqref{eq.dt.k} and
   \eqref{eq.by.parts.3p}
we get
\begin{eqnarray}
\label{eq.by.parts.22}
\lefteqn{
   2\,
   ( (- \varDelta)^{\frac{j}{2}} (\partial_t - \mu \varDelta) u + \mathbf{D} u + \nabla p) (\cdot,t),
     (- \varDelta)^{\frac{j+2}{2}} u (\cdot, t)
   )_{\mathbf{L}^2   }
}
\nonumber
\\
 \!\! & \!\! = \!\! & \!\!
   \frac{d}{dt} \| \nabla^{j\!+\!1} u (\cdot,t) \|^2_{\mathbf{L}^2   } \!\!
 + 2 \mu \| \nabla^{j\!+\!2} u (\cdot,t) \|^2_{\mathbf{L}^2   } \!\!
 + 2 ( (- \varDelta)^{\frac{j}{2}} \mathbf{D} u (\cdot,t), (- \varDelta)^{\frac{j\!+\!2}{2}} u (\cdot, t)
     )_{\mathbf{L}^2   } \!\!
\nonumber
\\
\end{eqnarray}
for all $0 \leq j \leq k+1$.
Next, according to the H\"older inequality, we get
\begin{equation}
\label{eq.En.Est3b}
   2 |( (- \varDelta)^{\frac{j}{2}} \mathbf{D}  u, (- \varDelta)^{\frac{j+2}{2}} u )_{\mathbf{L}^2   }|
 \leq
   \frac{2}{\mu}\, \| (- \varDelta)^{\frac{j}{2}} \mathbf{D} u \|^2_{\mathbf{L}^2   }
 + \frac{\mu}{2}\, \| (- \varDelta)^{\frac{j+2}{2}} u (\cdot,t) \|_{\mathbf{L}^2   },
\end{equation}
and so
\begin{eqnarray}
\label{eq.En.Est3a}
\lefteqn{
   2\, ( (- \varDelta)^{\frac{j}{2}} f (\cdot,t), (- \varDelta)^{\frac{j+2}{2}} u (\cdot,t)
       )_{\mathbf{L}^{2}   }
}
\nonumber
\\
 & \leq &
   2\,  \| (- \varDelta)^{\frac{j}{2}} f (\cdot,t) \|_{\mathbf{L}^{2}   }
        \| (- \varDelta)^{\frac{j+2}{2}} u (\cdot,t) \|_{\mathbf{L}^{2}   }
\nonumber
\\
 & \leq &
   \frac{4}{\mu}\, \| (- \varDelta)^{\frac{j}{2}} f (\cdot,t) \|^2_{\mathbf{L}^{2}   }
 + \frac{\mu}{4}\, \| (- \varDelta)^{\frac{j+2}{2}} u (\cdot,t) \|^2_{\mathbf{L}^{2}   }
\nonumber
\\
\end{eqnarray}
for all $t \in [0,T]$.
By the H\"older inequality with
$
   \displaystyle
   q_1 = \frac{\mathfrak{r}}{3}
$
and
$
   \displaystyle
   q_2 = \frac{\mathfrak{r}}{\mathfrak{r}-3},
$
\begin{equation}
\label{eq.En.Est3x00}
   \int_0^t
   \| u (\cdot,s) \|^{2}_{\mathbf{L}^{2}   }
   \| u (\cdot,s) \|^{2}_{\mathbf{L}^{\mathfrak{r}}   }
   ds
 \leq
   \| u \|^{2}_{L^{\frac{2}{3} \mathfrak{r}} ([0,t], \mathbf{L}^2   )}
   \| u \|^{2}_{L^\mathfrak{s} ([0,t], \mathbf{L}^{\mathfrak{r}}   )}.
\end{equation}
On summarising inequalities
   \eqref{eq.du.solution},
   \eqref{eq.by.parts.22},
   \eqref{eq.En.Est3b},
   \eqref{eq.En.Est3bbb},
   \eqref{eq.En.Est3x00} and
   \eqref{eq.En.Est3a}
we immediately obtain
\begin{eqnarray}
\label{eq.En.Est3x0}
\lefteqn{
   \| \nabla^{j+1} u (\cdot,t) \|^2_{\mathbf{L}^2   }
 + \mu \int_0^{t} \| \nabla^{j+2} u (\cdot,s) \|^2_{\mathbf{L}^{2}   } ds
}
\nonumber
\\
 & \leq &
   \| \nabla^{j+1} u_0 \|^2_{\mathbf{L}^{2}   }
 + \frac{4}{\mu} \| \nabla^{j} f \|^2_{L^2 (I,\mathbf{L}^2   )}
 + c (j, \mathfrak{s}, \mathfrak{r})
   \| u \|^{2}_{L^{\frac{2 \mathfrak{r}}{n}} ([0,t], \mathbf{L}^2   )}
   \| u \|^{2}_{L^\mathfrak{s} ([0,t],\mathbf{L}^{\mathfrak{r}}   )}
\nonumber
\\
 & + &
   c (j, \mathfrak{s}, \mathfrak{r})
   \frac{1}{\mu}
   \int_0^t \| u (\cdot,s) \|^{\mathfrak{s}}_{\mathbf{L}^{\mathfrak{r}}   }
            \| \nabla^{j+1} u (\cdot,s) \|^2_{\mathbf{L}^2   } ds
 + c (j, \mathfrak{s}, \mathfrak{r})
   \| u \|^{2}_{\mathbf{L}^{2}   }
\nonumber
\\
\end{eqnarray}
for all $t \in [0,T]$.
It is worth to be mentioned that the constants need not be the same in diverse applications.
By
   \eqref{eq.En.Est2},
   \eqref{eq.En.Est2add}
and
   \eqref{eq.En.Est3x0},
given any $0 \leq j \leq k+1$, we get an estimate
\begin{eqnarray}
\label{eq.En.Est3x}
\lefteqn{
   \| \nabla^{j+1} u (\cdot,t) \|^2_{\mathbf{L}^2   }
 + \mu \int_0^{t} \| \nabla^{j+2} u (\cdot,s) \|^2_{\mathbf{L}^{2}   } ds
}
\nonumber
\\
 & \leq &
   \| (f,u_0) \|^2_{j+1,\mu,T}
 + c (j, \mathfrak{s}, \mathfrak{r}) T^{\frac{3}{\mathfrak{r}}}
   \| (f,u_0) \|^2_{0,\mu,T}
   \| u \|^{2}_{L^\mathfrak{s} ([0,t],\mathbf{L}^{\mathfrak{r}}   )}
\nonumber
\\
 & + &
   c (j, \mathfrak{s}, \mathfrak{r}) \frac{1}{\mu}
   \int_0^t \| u (\cdot,s) \|^{\mathfrak{s}}_{\mathbf{L}^{\mathfrak{r}}   }
            \| \nabla^{j+1} u (\cdot,s) \|^2_{\mathbf{L}^2    } ds
 + c (j, \mathfrak{s}, \mathfrak{r}) T\, \| (f,u_0) \|^2_{0,\mu,T}
\nonumber
\\
\end{eqnarray}
for all $t \in I$.

On applying Gronwall's Lemma \ref{l.Perov}
%\ref{l.Groenwall} 
to \eqref{eq.En.Est3x} with
\begin{eqnarray*}
   A (t)
 \! & \! = \! & \!
   \| (f,u_0) \|^2_{j+1,\mu,T}
 + \left( c (j, \mathfrak{s}, \mathfrak{r}) T^{\frac{3}{\mathfrak{r}}}
          \| u \|^{2}_{L^\mathfrak{s} ([0,t],\mathbf{L}^{\mathfrak{r}}   )}
        + c (j, \mathfrak{s}, \mathfrak{r}) T
   \right)
   \| (f,u_0) \|^2_{0,\mu,T},
\\
   Y (t)
 \! & \! = \! & \!
   \| \nabla^{j+1} u (\cdot,t) \|^2_{\mathbf{L}^2   },
\\
   B (t)
 \! & \! = \! & \!
   c (j, \mathfrak{s}, \mathfrak{r}) \frac{1}{\mu}
   \| u (\cdot,t) \|^{\mathfrak{s}}_{\mathbf{L}^{\mathfrak{r}}   }
\end{eqnarray*}
we conclude that,
   for all $t \in [0,T]$ and $0 \leq j \leq k+1$,
\begin{equation}
\label{eq.d.sup}
   \| \nabla^{j+1} u (\cdot,t) \|^2_{\mathbf{L}^2   }
 \leq
   c (j, \mathfrak{s}, \mathfrak{r}, T, \mu, (f,u_0))
   \exp
   \Big(
   c (j, \mathfrak{s}, \mathfrak{r}) \frac{1}{\mu}
   \int_0^t
   \| u (\cdot,s) \|^{\mathfrak{s}}_{\mathbf{L}^{\mathfrak{r}}   }
   ds
   \Big)
\end{equation}
with a positive constant $c (j, \mathfrak{s}, \mathfrak{r}, T, \mu, (f,u_0))$ independent of $u$.
Obviously, \eqref{eq.En.Est3x} and \eqref{eq.d.sup} imply the first estimate of \eqref{eq.En.Est4}.

Next, applying \eqref{eq.En.Est3bbb} and \eqref{eq.En.Est3x00} we see that
\begin{eqnarray*}
\lefteqn{
   \| (- \varDelta)^{\frac{j}{2}} \mathbf{D} u \|^2_{L^2 ([0,t], \mathbf{L}^2   )}
}
\\
 & \leq &
   \| \nabla^{j+2} u \|^2_{L^2 ([0,t], \mathbf{L}^2   )}
 + c (j, \mathfrak{s}, \mathfrak{r},\varepsilon\!=\!1)\,
   \| u \|^{\mathfrak{s}}_{L^\mathfrak{s} ([0,t],\mathbf{L}^{\mathfrak{r}}   )}
   \| \nabla^{j+1} u \|^2_{C ([0,t], \mathbf{L}^2   )}
\\
 & + &
   2 c (j, \mathfrak{r})\,
   \| u \|^{2}_{L^{\frac{2 \mathfrak{r}}{3}} ([0,t], \mathbf{L}^2   )}
   \| u \|^{2}_{L^\mathfrak{s} ([0,t], \mathbf{L}^{\mathfrak{r}}   )}
 + 2 c (j, \mathfrak{r})\, \| u \|^{2}_{L^{2} ([0,t], \mathbf{L}^2   )},
\end{eqnarray*}
the constants being independent of $u$.
So, the second estimate of \eqref{eq.En.Est4} follows from
   \eqref{eq.En.Est2} and
   \eqref{eq.En.Est4}.

We are now ready to establish the desired estimates on $\partial_t u$ and $p$.
Indeed, since $\mathrm{div}\, u = 0$, we get
\begin{equation}
\label{eq.dtu+dp}
   \| (- \varDelta)^{\frac{j}{2}} (\partial_t u + \nabla p) \|^2_{\mathbf{L}^{2}   }
 = \| \nabla^{j} \partial_t u \|^2_{\mathbf{L}^{2}   }
 + \| \nabla^{j+1} p \|^2_{\mathbf{L}^{2}   }
\end{equation}
for all $j$ satisfying $0 \leq j \leq k+1$.
From \eqref{eq.du.solution} it follows that
\begin{eqnarray}
\label{eq.p+u}
\lefteqn{
   \frac{1}{2}\,
   \| (- \varDelta)^{\frac{j}{2}} (\partial_t u + \nabla p) \|^2_{L^2 (I,\mathbf{L}^2   )}
}
\nonumber
\\
 & \leq &
   \| \nabla^{j} f \|^2_{L^2 (I,\mathbf{L}^2   )}
 + \mu\, \| \nabla^{j+2} u \|^2_{L^2 (I, \mathbf{L}^2   )}
 + \| (- \varDelta)^{\frac{j}{2}} \mathbf{D} u \|^2_{L^2 (I, \mathbf{L}^2   )}
\nonumber
\\
\end{eqnarray}
for all $0 \leq j \leq k+1$.
Therefore, the third estimate of \eqref{eq.En.Est4} follows from
   the first and second estimates of \eqref{eq.En.Est4},
   \eqref{eq.dtu+dp} and
   \eqref{eq.p+u},
showing the lemma.
\end{proof}

Clearly, we may obtain additional information on $\partial_t u$ and $p$.

\begin{lemma}
\label{c.En.Est.g.k1}
Under the hypotheses of Lemma \ref{t.En.Est.g.2},
\begin{equation}
\label{eq.En.EstD.C}
\begin{array}{rcl}
   \| \nabla^{j} \mathbf{D} u \|_{C (I, \mathbf{L}^2   )}
 & \leq &
   c_j ( (f,u_0), u),
\\
   \| \nabla^{j} \partial_t u \|^2_{C (I, \mathbf{L}^2   )}
 + \| \nabla^{j+1} p \|^2_{C (I, \mathbf{L}^2   )}
 & \leq &
   c_j ( (f,u_0), u)
\end{array}
\end{equation}
for all $0 \leq j \leq k$, with a positive constant $c_j ( (f,u_0), u)$ depending on the
norms
   $\| (f,u_0) \|_{0,\mu,T}, \ldots, \| (f,u_0) \|_{k+2,\mu,T}$,
   $\| \nabla^{j} f \|_{C (I, \mathbf{L}^{2}   )}$
and
   $\| u \|_{L^\mathfrak{s} (I, \mathbf{L}^\mathfrak{r}   )}$.
\end{lemma}

As mentioned, the constants on the right-hand side of \eqref{eq.En.EstD.C} may also depend
on $\mathfrak{s}$, $\mathfrak{r}$, $T$, $\mu$, etc., but we do not display this dependence
in notation.

\begin{proof}
Using \eqref{eq.du.solution}, we get
\begin{eqnarray}
\label{eq.En.Est.41.tp+}
\lefteqn{
   \sup_{t \in [0,T]}
   \| (- \varDelta)^{\frac{j}{2}} (\partial_t u + \nabla p) (\cdot, t) \|^2_{\mathbf{L}^{2}   }
}
\nonumber
\\
 & \leq &
   \sup_{t \in [0,T]}
   \| (- \varDelta)^{\frac{j}{2}} (f + \mu \varDelta u + \mathbf{D} u) (\cdot, t) \|^2_{\mathbf{L}^{2}   }
\nonumber
\\
 & \leq &
   2
   \sup_{t \in [0,T]}
   \left( \| \nabla^{j} f (\cdot, t) \|^2_{\mathbf{L}^{2}   }
        + \| \nabla^{j+2} u (\cdot, t) \|^2_{\mathbf{L}^{2}   }
        + \| \nabla^{j} \mathbf{D} u (\cdot, t) \|^2_{\mathbf{L}^{2}   }
   \right)
\nonumber
\\
\end{eqnarray}
for all $0 \leq j \leq k$.
The first two summands in the last line of \eqref{eq.En.Est.41.tp+} can be estimated via the data
   $(f,u_0)$ and
   $\| u \|_{L^\mathfrak{s} (I,\mathbf{L}^\mathfrak{r}   )}$
using Lemma \ref{t.En.Est.g.2}.

On applying Lemma \ref{l.En.Est.Du.L2} to the third summand in \eqref{eq.En.Est.41.tp+} we see that
\begin{eqnarray}
\label{eq.Dk-2D}
\lefteqn{
   \| \nabla^j \mathbf{D} u \|^2_{C (I, \mathbf{L}^2   )}
}
\nonumber
\\
 & \leq &
   \| \nabla^{j+2} u \|^2_{C (I, \mathbf{L}^2   )}
 + c (j, \mathfrak{s}, \mathfrak{r}, \varepsilon\!=\!1)\,
   \| u \|^{\mathfrak{s}}_{C (I, \mathbf{L}^{\mathfrak{r}}   )}
   \| \nabla^{j+1} u \|^2_{C (I, \mathbf{L}^2   )}
\nonumber
\\
 & + &
   c (j, \mathfrak{s}, \mathfrak{r})\,
   \| u \|^{2}_{C (I, \mathbf{L}^2   )}
   \| u \|^{2}_{C (I, \mathbf{L}^{\mathfrak{r}}   )}
 + c (j, \mathfrak{s}, \mathfrak{r})\, \| u \|^{2}_{C (I, \mathbf{L}^2   )}
\nonumber
\\
\end{eqnarray}
for all $0 \leq j \leq k$, the constants being independent of $u$. 
On the other hand, we may use the Sobolev embedding theorem (see for instance
   \cite[Ch.~4, Theorem 4.12]{Ad03} or
   \eqref{eq.Sob.index},
   \eqref{eq.Sob.indexc})
to conclude that for any $\lambda \in [0,1/2)$ there exists a constant $c (\lambda)$ independent of
$u$ and $t$,  such that
$$
    \| u (\cdot,t) \|_{\mathbf{C}^{0,\lambda}   }
 \leq
   c (\lambda)\,  \| u (\cdot,t) \|_{\mathbf{H}^{2}   }
$$
for all $t \in [0,T]$.
Then energy estimate \eqref{eq.En.Est2} and Lemma \ref{t.En.Est.g.2} imply immediately that
\begin{equation}
\label{eq.Sob.0}
   \sup_{t \in [0,T]} \|u (\cdot,t) \|_{\mathbf{C}^{0,\lambda}   }
 \leq
    c ((f,u_0), u),
\end{equation}
where the constant $c ((f,u_0), u)$ depends on
   $\| (f,u_0) \|_{j',\mu, T}$ with $j' = 0, 1, 2$
and
   $\| u \|_{L^\mathfrak{s} (I, \mathbf{L}^\mathfrak{r}   )}$,
if inequality \eqref{eq.Sob.index} is fulfilled.
In particular,
\begin{equation}
\label{eq.CT.LsLr}
   \| u \|^{\mathfrak{s}}_{C (I, \mathbf{L}^{{\mathfrak r}}   )}
 \leq
   T \ell^{\frac{3 \mathfrak{s}}{\mathfrak{r}}}
   \sup_{t \in [0,T]} \| u (\cdot,t) \|^{\mathfrak{s}}_{\mathbf{C}   }
 \leq
    T \ell^{\frac{3 \mathfrak{s}}{\mathfrak{r}}}
    c ((f,u_0), u)
\end{equation}
with constant $c ((f,u_0), u)$ from \eqref{eq.Sob.0}. 
Hence, the first estimate of \eqref{eq.En.EstD.C} is fulfilled. 

At this point Lemma \ref{t.En.Est.g.2} and
   \eqref{eq.En.Est2add},
   \eqref{eq.En.Est.41.tp+},
   \eqref{eq.Dk-2D} and
   \eqref{eq.CT.LsLr}
allow us to conclude that
\begin{equation}
\label{eq.En.Est.41.tp++}
   \sup_{t \in [0,T]}
   \| (- \varDelta)^{\frac{j}{2}} (\partial_t u + \nabla p) (\cdot, t) \|^2_{\mathbf{L}^{2}   }
 \leq
   c (j, (f,u_0), u)
\end{equation}
for all $j = 0, 1, \ldots, k$, where $c (j, (f,u_0), u)$ is a positive constant depending on
   $\| (f,u_0) \|_{j',\mu,T}$ with $0 \leq j' \leq k+2$,
   $\| u \|_{L^\mathfrak{s} (I, \mathbf{L}^\mathfrak{r}   )}$
and
   $T$.
Hence, the second estimate of \eqref{eq.En.EstD.C} follows from
   \eqref{eq.dtu+dp} and
   \eqref{eq.En.Est.41.tp++}.
\end{proof}

Our next objective is to evaluate the higher derivatives of both $u$ and $p$ with respect to 
$x$ and $t$.

\begin{lemma}
\label{c.En.Est.g.ks}
Suppose that
   $s\in \mathbb{N}$, 
   $k \in \mathbb{Z}_+$
and
   $\mathfrak{s}$, $\mathfrak{r}$ satisfy \eqref{eq.s.r}.
If
   $(u,p) \in B^{k,2s,s}_\mathrm{vel} (I) \times B^{k+1,2(s-1),s-1}_\mathrm{pre} (I)$
is a solution to the  Navier-Stokes equations of \eqref{eq.NS.map}  
 with data
   $(f,u_0) \in B^{k,2(s-1),s-1}_\mathrm{for} (I) \times V_{k+2s}$
then it is subjected to an estimate of the form
\begin{equation}
\label{eq.En.Est.Bks}
   \| (u,p) \|_{B^{k,2s,s}_\mathrm{vel} (I) \times B^{k+1,2(s-1),s-1}_\mathrm{pre} (I)}
 \leq
   c (k, s, (f,u_0), u),
\end{equation}
the constant on the right-hand side depending on
   $\| f \|_{B^{k,2(s-1),s-1}_\mathrm{for} (I)}$,
   $\| u_0 \|_{V_{2s+k}}$
and
   $\| u \|_{L^\mathfrak{s} (I, \mathbf{L}^\mathfrak{r}   )}$
as well as on
   $\mathfrak{r}$, $T$, $\mu$, etc.
\end{lemma}

\begin{proof}
For $s=1$ and any $k \in \mathbb{Z}_+$, the statement of the lemma was proved in Lemmata
   \ref{t.En.Est.g.2} and
   \ref{c.En.Est.g.k1}.

Then the statement follows by induction with respect to $s$ from the recurrent formulas
\begin{equation}
\label{eq.recurrent}
\begin{array}{rcl}
   \partial^\alpha \partial^{j}_t (\partial_t u + \nabla p)
 & =
 & \partial^\alpha \partial^{j}_t (f + \mu \varDelta u - \mathbf{D} u),
\\
   \| \partial^\alpha \partial^{j}_t (\partial_t u  +  \nabla p) \|^2_{\mathbf{L}^{2}   }
 & =
 & \| \partial^\alpha \partial^{j+1}_t u \|^2_{\mathbf{L}^{2}   }
 + \| \partial^\alpha \partial^{j}_t \nabla p \|^2_{\mathbf{L}^{2}   }
\end{array}
\end{equation}
provided that
   $\mathrm{div}\, u = 0$
and
   $j \in \mathbb{Z}_+$,
   $\alpha \in \mathbb{Z}^n_+ $
are fit for the assumptions.

Indeed, suppose the assertion of the lemma is valid for $s = s_0$ and any $k \in \mathbb{Z}_+$.
We then prove that it is fulfilled for $s = s_0+1$ and any $k \in {\mathbb Z}_+$.
As
$$
\begin{array}{rcl}
   (u,p)
 & \in
 & B^{k,2(s_0+1),s_0+1}_\mathrm{vel} (I) \times B^{k+1,2s_0,s_0}_\mathrm{pre} (I),
\\
   (f,u_0)
 & \in
 & B^{k,2s_0,s_0}_\mathrm{for} (I) \times V_{2(s_0+1)+k},
\end{array}
$$
then, by the definition of the spaces,
$$
\begin{array}{rcl}
   (u,p)
 & \in
 & B^{k+2,2s_0,s_0}_\mathrm{vel} (I) \times B^{k+3,2(s_0-1),s_0-1}_\mathrm{pre} (I),
\\
   (f,u_0)
 & \in
 & B^{k+2,2(s_0-1),s_0-1}_\mathrm{for} (I) \times V_{2s_0+(k+2)}.
\end{array}
$$
Thus, by the induction assumption,
\begin{equation}
\label{eq.En.Est.Bks.prime}
   \| (u,p) \|_{B^{k+2,2s_0,s_0}_\mathrm{vel} (I) \times B^{k+1,2(s_0-1),s_0-1}_\mathrm{pre} (I)}
 \leq
   c (k, s_0, (f,u_0), u),
\end{equation}
where the properties of the constant $c (k, s_0, (f,u_0), u)$ are similar to those described in
the statement of the lemma.

On the other hand, it follows from the first equality of \eqref{eq.recurrent} that for all suitable
$j$ we get
\begin{eqnarray}
\label{eq.est.ind.1}
\lefteqn{
   \| \nabla^{j} \partial^{s_0}_t (\partial_t u + \nabla p) \|^2_{\mathbf{L}^{2}   }
}
\nonumber
\\
 & = &
   \| \nabla^{j} \partial^{s_0}_t (f + \mu \varDelta u - \mathbf{D} u) \|^2_{\mathbf{L}^{2}   }
\nonumber
\\
 & \leq &
   2
   \left( \| \nabla^{j} \partial^{s_0}_t f \|^2_{\mathbf{L}^{2}   }
 + \mu\, \| \nabla^{j+2} \partial^{s_0}_t u \|^2_{\mathbf{L}^{2}   }
 + \| \nabla^{j} \partial^{s_0}_t \mathbf{D} u \|^2_{\mathbf{L}^{2}   }
   \right).
\nonumber
\\
\end{eqnarray}
By the induction assumption, if
   $0 \leq j \leq k+1$ and
   $0 \leq i \leq k$,
then the norms
$
   \| \nabla^{j} \partial^{s_0}_t f \|^2_{L^2 (I, \mathbf{L}^{2}   )}
$
and
$
   \| \nabla^{i} \partial^{s_0}_t f \|^2_{C (I, \mathbf{L}^{2}   )}
$
are finite and
\begin{equation}
\label{eq.est.ind.2}
\begin{array}{rcl}
   \| \nabla^{j+2} \partial^{s_0}_t u \|^2_{L^2 (I, \mathbf{L}^{2}   )}
 & \leq
 & c\, \| u \|^2_{B^{k+2,2s_0,s_0}_\mathrm{vel} (I)},
\\
   \| \nabla^{i+2} \partial^{s_0}_t u \|^2_{C (I, \mathbf{L}^{2}   )}
 & \leq
 & c\, \| u \|^2_{B^{k+2,2s_0,s_0}_\mathrm{vel} (I)}
\end{array}
\end{equation}
with constants $c$ independent of $u$ and not necessarily the same in diverse applications.
Besides, \eqref{eq.Leib.k.s} and \eqref{eq.B.cont.9}, \eqref{eq.B.cont.10} with $w = u$ yield
\begin{equation}
\label{eq.est.ind.3}
\begin{array}{rcl}
   \| \nabla^{j} \partial^{s_0}_t \mathbf{D} u \|^2_{L^2 (I, \mathbf{L}^{2}   )}
 & \leq
 & c\, \| u \|^4_{B^{k+2,2s_0,s_0}_\mathrm{vel} (I)},
\\
   \| \nabla^{i} \partial^{s_0}_t \mathbf{D} u \|^2_{C (I, \mathbf{L}^{2}   )}
 & \leq
 & c\, \| u \|^4_{B^{k+2,2s_0,s_0}_\mathrm{vel} (I)}
\end{array}
\end{equation}
provided
   $0 \leq j \leq k+1$ and
   $0 \leq i \leq k$,
the constants being independent of $u$.

Finally, combining
   \eqref{eq.En.Est.Bks.prime},
   \eqref{eq.est.ind.1},
   \eqref{eq.est.ind.2},
   \eqref{eq.est.ind.3}
with the second equality of \eqref{eq.recurrent}, we conclude that
$$
   \| (u,p) \|_{B^{k,2(s_0+1),s_0+1}_\mathrm{vel} (I) \times B^{k+1,2s_0,s_0}_\mathrm{pre} (I)}
 \leq
   c (k, s_0+1, (f,u_0), u),
$$
where the constant on the right-hand side depends on
   $\| f \|_{B^{k,2s_0,s_0}_\mathrm{for} (I)}$,
   $\| u_0 \|_{V_{2(s_0+1)+k}}$
and
   $\| u \|_{L^\mathfrak{s} (I, \mathbf{L}^\mathfrak{r}   )}$
as well as on
   $\mathfrak{r}$, $T$, $\mu$, etc.
This proves the lemma.
\end{proof}

Keeping in mind Corollary \ref{c.clopen}, we are now in a position to show that the range of 
mapping \eqref{eq.map.A} is closed  if given subset   $S = S_\mathrm{vel} \times 
S_\mathrm{pre}$ of the product
   $B^{k,2s,s}_\mathrm{vel} (I) \times B^{k+1,2(s-1),s-1}_\mathrm{pre} (I)$ 
such that the image $\mathcal{A} (S)$ is precompact in the space
   $B^{k,2(s-1),s-1}_\mathrm{for} (I) \times V_{2s+k}$,
the set $S_\mathrm{vel}$ is bounded in the space
   $L^{\mathfrak{s}} (I,\mathbf{L}^{\mathfrak{r}}   )$ with 
	a pair $\mathfrak{s}$, $\mathfrak{r}$ satisfying \eqref{eq.s.r}.

Indeed, let a pair
    $(f,u_0) \in B^{k,2(s-1),s-1}_\mathrm{for} (I) \times V_{2s+k}$
belong to the closure of the range of values of the mapping $\mathcal{A}$.
Then there is a sequence $\{ (u_i,p_i) \}$ in
    $B^{k,2s,s}_\mathrm{vel} (I) \times B^{k+1,2(s-1),s-1}_\mathrm{pre} (I)$
such that the sequence
   $\{ (f_i, u_{i,0}) = \mathcal{A} (u_i, p_i) \}$
converges to  $(f,u_0)$ in the space $B^{k,2(s-1),s-1}_\mathrm{for} (I) \times V_{2s+k}$.

Consider the set $S = \{ (u_i, p_i)\}$.
As the image $\mathcal{A} (S) = \{ (f_i, u_{i,0}) \}$ is precompact in
   $B^{k,2(s-1),s-1}_\mathrm{for} (I) \times V_{2s+k}$,
it follows from our assumption that the subset
   $S_\mathrm{vel} = \{ u_i \}$ of
   $B^{k,2s,s}_\mathrm{vel} (I)$
is bounded in the space  
$L^\mathfrak{s} (I, \mathbf{L}^{\mathfrak{r}}   )$. 

On applying Lemmata \ref{p.En.Est.u.strong} and \ref{c.En.Est.g.ks} we conclude that the sequence
   $\{ (u_i, p_i) \}$
is bounded in the space
   $ B^{k,2s,s}_\mathrm{vel} (I) \times B^{k+1,2(s-1),s-1}_\mathrm{pre} (I)$.
By the definition of $B^{k,2s,s}_\mathrm{vel} (I)$, the sequence $\{ u_i \} $ is bounded in
   $C (I, \mathbf{H}^{k+2s}   )$ and
   $L^2 (I, \mathbf{H}^{k+2s+1}   )$,
and the partial derivatives $\{ \partial_t^j u_i \}$ in time with $1 \leq j \leq s$ are 
bounded in
   $C (I, \mathbf{H}^{k+2(s-j)}   )$ and
   $L^2 (I, \mathbf{H}^{k+2(s-j+1)}   )$.
Therefore, there is a subsequence $\{ u_{i_k} \}$ such that

1)
The sequence
   $\{ \partial^{\alpha+\beta}_x \partial_t^j u_{i_k} \}$
converges weakly in $L^2 (I, \mathbf{L}^{2}   )$ provided that
   $|\alpha| + 2j \leq 2s$ and
   $|\beta| \leq k+1$.

2)
The sequence
   $\{ \partial^{\alpha+\beta}_x \partial_t^j u_{i_k} \}$
converges weakly-$^\ast$ in $L^\infty (I, \mathbf{L}^{2}   )$ provided that
   $|\alpha| + 2j \leq 2s$ and
   $|\beta| \leq k$.
	
It is clear that the limit $u$ of  $\{ u_{i_k} \}$ is a solution to the 
Navier-Stokes equations \eqref{eq.NS.weak} such that

1)
Each derivative
   $\partial^{\alpha+\beta}_x \partial_t^j u$
belongs to $L^2 (I, V_0)$ provided that
   $|\alpha| + 2j \leq 2s$ and
   $|\beta| \leq k+1$.

2)
Each derivative
   $\partial^{\alpha+\beta}_x \partial_t^j u$
belongs to $L^\infty (I, V_0)$ provided that
   $|\alpha| + 2j \leq 2s$ and
   $|\beta| \leq k$.
	
According to Theorem \ref{t.NS.unique} such a solution is unique. In addition, if 
\begin{equation}
\label{eq.0js-1}
\begin{array}{rcl}
   0 \, \leq \, j
 & \leq
 & s-1,
\\
   |\alpha| + 2j
 & \leq
 & 2s,
\\
   |\beta|
 & \leq
 & k,
\end{array}
\end{equation}
then
   $\partial^{\alpha+\beta}_x \partial_t^j u \in  L^2 (I, V_1)$ and
   $\partial^{\alpha+\beta}_x \partial_t^{j+1} u \in L^2 (I, V'_1)$.
Applying Lemma \ref{l.Lions} we readily conclude that
   $\partial^{\alpha+\beta}_x  \partial_t^j u \in C (I, V_0)$
for all $j$ and $\alpha$, $\beta$ satisfying (\ref{eq.0js-1}). 
Hence  it follows that $u$ belongs to the space
   $B^{k+2,2(s-1),s-1}_\mathrm{vel} (I)$.
Moreover, using formulas \eqref{eq.Leib.k.s} and \eqref{eq.B.cont.9}, \eqref{eq.B.cont.10} 
with $w=u$ implies that the derivatives
   $\partial^{\alpha+\beta}_x \partial_t^j \mathbf{D} u$
belong to $C (I, \mathbf{L}^2   )$ for all $j$ and $\alpha$, $\beta$ which satisfy
inequalities (\ref{eq.0js-1}).

Besides, according to Lemma \ref{l.projection.commute}, the operator $\mathbf{P}$ maps
   $C (I,\mathbf{L}^2   )$ continuously into
   $C (I,\mathbf{L}^2   )$.
Therefore, since $u$ is a solution to \eqref{eq.NS.weak} we deduce that
$$
   \partial^{\beta}_x \partial_t^s  u
 = \partial^{\beta}_x \partial_t^{s-1} \mu \varDelta u
 - \partial^{\beta}_x \partial_t^{s-1} \mathbf{P} \mathbf{D} u
 + \partial^{\beta}_x \partial_t^{s-1} \mathbf{P} f
$$
belongs to $C (I,V_0)$ for all multi-indices $\beta$ such that $|\beta| \leq k$.
%is in $C (I,\mathbf{L}^2   )$ provided   $|\alpha| + 2j \leq 2 (s-1)$ and   $|\beta| \leq k$.
In other words, $u$ lies in
   $B^{k,2s,s}_\mathrm{vel} (I)$. 
Finally, applying Proposition \ref{c.Sob.d}, %curr
we conclude that there is $p \in B^{k+1,2(s-1),s-1}_\mathrm{pre} (I)$
such that 
$$
\nabla p
 = (I - \mathbf{P}) (f - \mathbf{D} u),
$$
i.e. the pair $(u,p)$ belongs to $B^{k,2s,s}_\mathrm{vel} (I)
\times B^{k+1,2(s-1),s-1}_\mathrm{pre} (I)$ and it is a solution to \eqref{eq.NS.map}.

Thus, we have proved that the image of the mapping in \eqref{eq.map.A} is closed. 
Then the statement of the theorem related to the surjectivity of the mapping follows from
  % Theorem \ref{t.open.NS.short} and
   Corollary \ref{c.clopen}.
%The estimate follows readily from Lemma \ref{p.En.Est.u.strong}.
\end{proof}

\begin{corollary}
\label{c.LsLr}
Let and the numbers $\mathfrak{r}$, $\mathfrak{s}$ satisfy \eqref{eq.s.r}. 
Then mapping \eqref{eq.map.A.Frechet} is surjective 
if and only if, given subset   $S = S_\mathrm{vel} \times S_\mathrm{pre}$ of the product
   $C^\infty (I,V_\infty) \times C^\infty (I,\dot{C}^{\infty}) $ 
such that the image $\mathcal{A} (S)$ is bounded in the space
   $C^\infty (I,\mathbf{C}^{\infty}   ) \times V_\infty$,
the set $S_\mathrm{vel}$ is bounded in the space
   $L^{\mathfrak{s}} (I,\mathbf{L}^{\mathfrak{r}}   )$.
\end{corollary}

\begin{proof} Clearly, for each $k \in {\mathbb Z}_+$
and $s \in \mathbb N$, the space $C^\infty (I,\mathbf{C}^{\infty}   ) \times V_\infty$ is 
embedded compactly into the space 
$B^{k,2(s-1),s-1}_\mathrm{for} (I) \times V_{2s+k}$. Thus, the statement 
follows immediately from Corollary \ref{c.clopen} and Theorem \ref{t.LsLr}.
\end{proof}

\bigskip

\textit{Acknowledgments\,}
The first author was supported by a grant of the Foundation for the advancement of theoretical
physics and mathematics ``BASIS.''

\end{document}